\documentclass[12 pt]{article}
\usepackage[utf8]{inputenc}
\usepackage[margin=1in]{geometry}
\usepackage[usenames,dvipsnames,svgnames,table]{xcolor}
\usepackage[english]{babel}
\usepackage[utf8]{inputenc}
\usepackage{amsmath}
\usepackage{mathrsfs}
\usepackage{wrapfig}
\usepackage[T1]{fontenc}
\usepackage{amsfonts}
\usepackage{lineno}
\usepackage{amssymb}
\usepackage{amsthm}
\usepackage{xcolor}
\usepackage{graphicx}
\usepackage{appendix}
\usepackage{multicol}
\usepackage{multirow}
\usepackage{comment}
\usepackage{array}
\usepackage{booktabs}
\usepackage{blindtext}
\usepackage{oubraces}
\usepackage{enumerate}

\usepackage[parfill]{parskip}
\usepackage{tikz}
\usetikzlibrary{calc, through, fit}
\usetikzlibrary{decorations.markings, decorations.pathmorphing}
\usetikzlibrary{arrows, backgrounds}
\usetikzlibrary{positioning}
\usetikzlibrary{intersections}
\tikzset{>=latex}

\usepackage{soul}  
\usepackage{subcaption}

\usepackage{hyperref}

\newcommand{\be}{\begin{enumerate}}
\newcommand{\ee}{\end{enumerate}}
\newcommand{\bi}{\begin{itemize}}
\newcommand{\ei}{\end{itemize}}

\newcommand{\mc}{\mathcal}

\newtheorem{theorem}{Theorem}

\newtheorem{corollary}[theorem]{Corollary}

\newtheorem{lemma}[theorem]{Lemma}

\newtheorem{remark}{Observation}

\theoremstyle{comment}
\newtheorem*{acomment}{\color{BrickRed}{Comment}}

\title{Semicubic cages and small graphs of even girth from voltage graphs}
\author{Flor Aguilar$^{1}$, Gabriela Araujo-Pardo$^{1}$ and Leah Wrenn Berman$^{2}$. }
\begin{document}
\maketitle

\begin{abstract}
An \emph{$(3,m;g)$ semicubic graph} is a graph in which all vertices have degrees either $3$ or $m$ and fixed girth $g$. In this paper, we construct families of semicubic graphs of even girth and small order using two different techniques. The first technique generalizes a previous construction which glues cubic cages of girth $g$ together at remote vertices (vertices at distance at least $g/2$). The second technique, the main content of this paper, produces bipartite semicubic $(3,m; g)$-graphs with fixed even girth $g = 4t$ or $4t+2$ using voltage graphs over $\mathbb{Z}_{m}$. When $g = 4t+2$, the graphs have two vertices of degree $m$, while when $g = 4t$ they have exactly three vertices of degree $m$ (the remaining vertices are of degree $3$ in both cases).  Specifically, we describe infinite families of semicubic graphs $(3,m; g)$ for $g = \{6, 8, 10, 12\}$ for infinitely many values of $m$. The cases $g = \{6,8\}$ include the unique $6$-cage and the unique $8$-cage when $m = 3$. 

The families obtained in this paper for girth $g=\{10,12\}$ include examples with the best known bounds for semicubic graphs $(3,m; g)$.
\end{abstract}

\section{Introduction}
In this paper, we work with simple and finite graphs. We study a generalization of the \emph{Cage Problem}, which has been widely studied since cages were introduced by Tutte \cite{T47} in 1947 and after Erd\"os and Sachs \cite{ES63} proved their existence in 1963. An \emph{$(r,g)$-graph} is a $r$-regular graph in which the shortest cycle has length equal to $g$; that is, it is a $r$-regular graph with girth $g$. An \emph{$(r,g)$-cage} is an $(r,g)$-graph with the smallest possible number of vertices among all $(r,g)$-graphs. The Cage Problem consists of finding $(r,g)$-cages; it is well known that cages exist only for very limited sets of parameter pairs $(r, g)$. In the case where the orders of the $(r, g)$-cages match a simple lower bound due to Moore \cite{EJ13}, the cages are called \emph{Moore cages}. 

Biregular graphs, denoted as \emph{$(r,m;g)$-graphs}, generalize $(r,g)$-graphs, with cages generalizing to \emph{biregular cages}. Specifically, given three positive integers $r,m, g$ with $2\le r<m$, an \emph{$(r,m;g)$-graph} is a graph of girth $g$ in which all vertices have degrees $r$ and $m$. We denote the number of vertices of an $(r, m;g)$-graph as $n(r,m;g)$, and a \emph{biregular cage} is an $(r,m;g)$-graph in which $n(r,m;g)$ is as small as possible. Biregular graphs have been studied by many authors (see \cite{ABGMV08,ABV09,ABLM13,AEJ16,CGK81,DGM81,HWJ92,KPW77,W82,YL03}) since Chartrand, Gould and Kapoor \cite{CGK81} proved their existence. 
A biregular graph with $r = 3$ is often called a \emph{semicubic graph}, and naturally, a semicubic $(3,m;g)$-graph with a minimal number of vertices for a fixed $m>3$ and fixed $g$ is called a semicubic cage. 

In this paper, we construct families of semicubic graphs of even girth and small order using two different techniques. The first technique generalizes a construction used in \cite{ABV09,ABLM13} in which cubic cages of girth $g$ are glued together using \emph{remote vertices}, that is vertices at distance $g/2$. The second technique, which is the main content of this paper, consists of constructing semicubic graphs of even girth using voltage graphs. With this technique, we improve the graphs given using the ``identifying remote vertices'' technique for girth $g=\{6,8,10,12\}$. However, graphs with the same orders as those from our voltage graph construction were obtained previously for girth $g=\{6,8\}$ (in \cite{ABLM13,AEJ16,HWJ92}) using different techniques. Our principal contribution is for graphs of girth $g=\{10,12\}$, where we find new graphs with orders between the lower bounds given in \cite{AEJ16} and the upper bounds given in this paper found by identifying remote vertices.

The voltage graph construction gives us, naturally, the Heawood graph or $(3;6)$-cage and the Tutte graph or $(3;8)$-cage, for $m=3$ and $g=\{6,8\}$ respectively.  
These graphs occur as part of our constructions of families of $(3,m;6)$-cages of order $4m+2$ and $(3,m;8)$-graphs of order $9m+3$. We will detail how our constructions generalize the constructions given in \cite{ABLM13,HWJ92} in the corresponding sections. As the authors state in \cite{ABLM13}, for girth $8$ and $m=\{4,5,6,7\}$ these graphs, and also ours, are cages, while for the rest of the values of $m$ they are close to the lower bound $n(3,m;8)\geq \lceil \frac{25m}{3}\rceil+5$ given for $m\geq 7$.

In \cite{AEJ16}, the authors proved that for $m$ much larger than $r$ and even girth $g\equiv 2\mod{4}$ there exist infinite families of $(r,m;g)$-graphs with few vertices, with order close to the lower bound also given in that paper. Specifically, for girth $g=6$, the graphs described are biregular cages, because they attain the lower bound given in \cite{YL03}. However, in that paper, the authors did not give an explicit construction of these graphs; they only proved their existence using a strong result about Hamiltonian graphs and girths given by Sachs in 1963 (\cite{S63}). In particular, for girth $10$ the existence of a semiregular cage continues to be open for small values of $m$.

For girth $g=10$, using the identifying remote vertices technique, we obtain graphs of order greater than $22m+2m/3$ (see Section \ref{identify}).  With the voltage graph construction, we give explicit constructions for two different infinite families of $(3,m;10)$-graphs. The first construction produces $(3,m;10)$-graphs for $m\geq 4$ of order $24m+2$, with  2 vertices of degree $m$ and  $24m$ vertices of degree 3. This number of vertices coincides with the parameter obtained in \cite{AEJ16} for $m$ much larger than $3$. The second construction produces graphs of order $20m+2$ for $m\geq 7$ with 2 vertices of degree $m$ and $20m$ vertices of degree 3. This second family clearly improves the upper bound for $(3,m;10)$-cages given in Section \ref{identify} and has a difference of less than $3m$ to the lower bound $n(3,m;10)\geq \lceil \frac{53m}{3}\rceil+9$ given in Lemma 3.4 in \cite{AEJ16}.  

Finally, for girth $12$, using the identifying remote vertices technique, we obtain graphs of order greater than $41m+m/3$ (see Section \ref{identify}). Using voltage graphs, we give explicit constructions of two different infinite families of  $(3,m;12)$-graphs. The first gives us $(3,m;12)$-graphs for $m\geq 9$ of order $49m+3$ with 3 vertices of degree $m$ and $49m$ vertices of degree 3. This construction is new and gives us new parameters of semicubic cages of girth $12$, but the order of these graphs is bigger than the upper bounds from the identifying vertices construction. However, it is based on a general structure that will be considered throughout the paper, and for this reason we consider that it is important to include it among our results. We also present a second family of semicubic graphs of girth $12$ using voltage graphs, giving us $(3,m;12)$-graphs for $m\geq 10$ of order $41m+3$ with 3 vertices of degree $m$ and $41m$ vertices of degree 3. This family improves the upper bound given in Section \ref{identify} and produces graphs with a difference of less than $5m$ to the lower bound $n(3,m;12)\geq \lceil \frac{109m}{3}\rceil+17$ given in Lemma 3.4 in \cite{AEJ16}.

This paper is organized as follows. In Section \ref{identify}, we construct semicubic graphs with few vertices identifying remote vertices that give us upper bounds for semicubic cages of even girth generalizing. In Section \ref{prel} we give some definitions of voltage graphs and derived graphs that we will use in the rest of the paper, including introducing a new definition of \emph{pinned} vertices. In Section  \ref{case4t+2}, we describe a general construction for graphs that are of girth at most $4t+2$, with two vertices of degree $m$ and many vertices of degree $3$, via lifting certain types of voltage graphs over $\mathbb{Z}_{m}$. We provide constructions for graphs with girths exactly $6$ and $10$,  including producing the Heawood graph as the $m= 3$ case of the girth $6$ family. In Section \ref{case4t}, we similarly 
describe a general construction for graphs with girth at most $4t$ with three vertices of degree $m$ (this general construction is used in \cite{A23} to construct bi-regular graphs of even girth $g=4t$, the paper is in progress.)
We provide explicit voltage graphs whose lifts form infinite families of semicubic graphs of girths equal to $8$ and $12$, including producing the Tutte $8$-cage as a member of the girth $8$ family, where $m = 3$. 
 
 \section{Constructing upper bounds for semicubic cages of even girth by identifying remote vertices}\label{identify}
In this section we generalize Theorem 3, given in \cite{ABV09}, in which the authors identify copies of $(r;g)$-cages at \emph{remote vertices}, which are vertices at distance at least $g/2$. These techniques are also used in \cite{ABLM13} to construct biregular graphs of even girth. 


The results obtained in \cite{ABV09,ABLM13}, identifying remote vertices on graphs of girth $8$, are better than the results given in Theorem \ref{thm:identify} for girth $8$. In particular, for girth $g=8$, there exist results obtained using the properties of  generalized quadrangles that produce five remote vertices in the $(3,8)$-cage. Using this fact,  Corollary 7 in  \cite{ABV09} states that the $(\{3,m\};8)$-cages have order $8m+\frac{m}{3}+5$ for $m=3k$ and $k\geq 1$, and  Corollary 3.3 in \cite{ABLM13} states that $n(\{3,m\};8)\leq 8m+\frac{m}{3}-\frac{16}{3}t+21$ for $m=3k+t$ and $t\in \{1,2\}$. 

Consequently, we will use  Theorem \ref{thm:identify} only to produce bounds of the order of semicubic cages of girth $g=\{10,12\}$, which are the parameters that we improve on this paper. The constructions described in the proof are illustrated in Figure \ref{fig:identify}.

\begin{theorem}\label{thm:identify}
Let $G$ be a $(3;g)$-graph of even girth and order $n_g$ with at least two vertices at distance $g/2$. If $m=3k+t$, we obtain $(3,m;g)$-graphs of order: 
 \[k(n_g-2) + 
 \begin{cases}
 2 & \text{ if } t = 0\\
 n_{g} +2 &\text{ if } t = 1\\
 n_{g} &\text{ if } t = 2
 \end{cases}
 \]
\end{theorem}

\begin{proof}
We divide the proof into three cases.
\begin{enumerate}[(i)]
\item  Let $m=3k$, and let  $G_{1}$ and $G_{2}$ be two copies of a $(3,g)$-graph of even girth $g$ and order $n_{g}$. Let $x_1$ and $y_1$ be two vertices at distance at least $g/2$ (remote vertices) in $G_1$ and let $x_2$ and $y_2$ be two remote vertices in $G_2$. Construct a graph $G$ by taking $G_1$ and $G_2$ and identifying  $x_1$ with $x_2$ (call this new vertex $x$) and $y_1$ with $y_2$ (called $y$). 
It is easy to see that the shortest cycle that passes through two vertices of $G_1$, with at least one of them different from either $x$ or $y$, is totally contained in $G_1$ and thus has length at least $g$ (and analogously for $G_{2}$). If the cycle contains both $x$ and $y$, then, since the  distance between $x$ and $y$ is at least $g/2$, the cycle is given by two disjoint paths (in $G_{1}$ or $G_{2}$) between $x$ and $y$, each of them with length at least $g/2$, so together they form a cycle of length at least $g$. 
Now let $G$ be a graph formed by identifying $k$ copies, where the $i$-th copy is labelled $G_{i}$, at pairs of remote vertices $x_{i}$ and $y_{i}$ in $G_{i}$, calling the identified vertices $x$ and $y$ as before. Since each of the graphs $G_{i}$ is 3-regular, the identified vertices $x$ and $y$ have degree $m = 3k$. Applying the same shortest cycle analysis as above to each pair $(G_{i}, G_{j})$, it follows that the girth of $G$ is also at least $g$,
and $G$ has order $k n_g-2(k-1)=k(n_g-2)+2$, with two vertices of degree $m$ and $k(n_{g} - 2)$ vertices of degree 3.
\item   Suppose that $m=3k+2$. Take $k+1$ copies of a  $(3,g)$-graph of order $n_{g}$ of even girth $g$, and label the $i$-th copy as $G_{i}$. In $G_{1}$, let $x_1y_1$ be any edge. Delete $x_1y_1$, and call this new graph $G'_1$.  Notice that all the vertices in $G'_1$ have degree $3$ except $x_1$ and $y_1$, which have degree $2$, and since $G_1$ has girth $g$,  $x_1$ and $y_1$ are now at distance at least $g-1$. Now, suppose that $x_i$ and $y_i$ are two vertices at distance at least $g/2$ in $G_i$, for $i\in \{2,\ldots,k+1\}$. Construct a new graph $G$ by identifying all the vertices $x_i$, calling the new vertex $x$, and all the vertices $y_i$, calling the new vertex $y$. As in the previous case, we obtain a graph of girth $g$, but in this case, the $n_{g} - 2$ vertices in each copy other than $x$ and $y$ have degree $3$, and $x$ and $y$ have degree $m=3k+2$. It follows that $G$ is a $(3,m;g)$-graph of order $(n_{g}-2)(k+1) + 2 = k(n_{g}-2) + n_{g}$.
\item  Suppose that $m=3k+1$. take $k+1$ copies of a  $(3,g)$-graph of order $n_{g}$ of even girth $g$, and label the $i$-th copy as $G_{i}$. As before, choose any edge $x_{1}y_{1}$ in $G_{1}$ and delete it. Now add two vertices to $G_{1}$, one of them a neighbor of $x_1$, called $x$, and the other a neighbor of $y_1$, called $y$. Notice that these two vertices are at a distance of at least $g+1$ in $G_{1}$. 
As in the previous cases, construct a graph $G$ identifying $x$ and $y$ with two remote vertices $x_i$ and $y_i$  in each of the $k$ remaining graphs $G_i$. This graph $G$ has two vertices of degree $3k+1$ and $k(n_{g} - 2)+n_{g}$ vertices of degree 3, for a total order of 
$k(n_g-2)+n_g+2$.
\end{enumerate}
\end{proof}

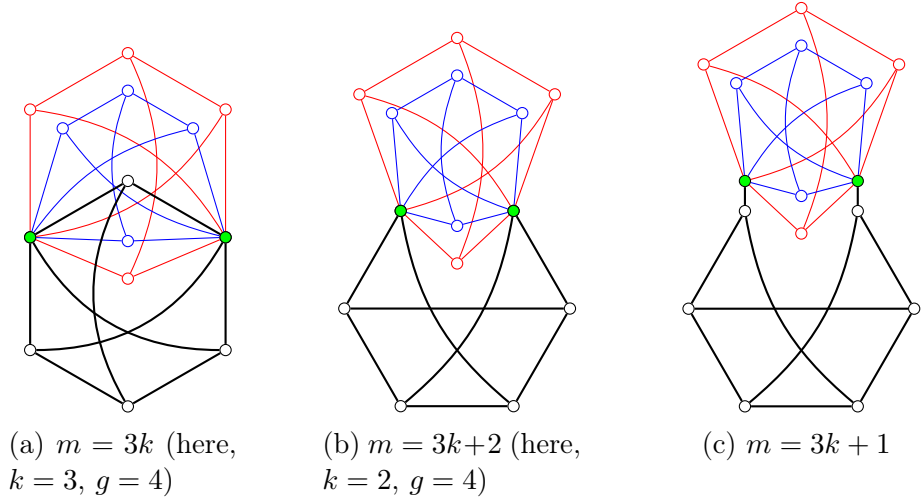
\begin{figure}[htbp]
\begin{center}
\subfloat[$m = 3k$ (here, $k = 3$, $g = 4$)]{
\begin{tikzpicture}[vtx/.style={draw, circle, inner sep = 1.5 pt, font = \tiny}]
\pgfmathsetmacro{\r}{1.5}
\foreach \i in {0,...,5}{
\node[vtx] (v\i) at (360*\i/6+30: \r){};
}
\foreach \i in {0,1,2,4}{
\node[red, vtx] (w\i) at ($(360*\i/6+30 : \r)+(0,1.7)$){};
\node[blue, vtx] (u\i) at ($(360*\i/6+30: \r-.5)+(0,1.7)$){};
}

\foreach \j/\k/\bend in {w/red/25, u/blue/-20}{
\draw[\k] (v0) -- (\j0);
\draw[\k] (\j1) -- (\j2);
\draw[\k] (\j2) -- (v2);
\draw[\k] (v2) -- (\j4);
\draw[\k] (\j4) -- (v0);
\draw[\k] (\j1) -- (\j0);
\draw[\k] (\j0) to[bend left =\bend](v2);
\draw[\k] (\j1) to[bend left = \bend](\j4);
\draw[\k] (\j2) to[bend left = \bend](v0);
}

\draw[thick] (v0) -- (v1) -- (v2) -- (v3) -- (v4) -- (v5) -- (v0);
\draw[thick] (v0) to[bend left = 30] (v3);
\draw[thick] (v1) to[bend left = -30] (v4);
\draw[thick] (v2) to[bend left = -30] (v5);

\node[vtx, fill = green]at (v2){};
\node[vtx, fill = green]at (v0){};

\end{tikzpicture}
}
\hspace{1cm}
\subfloat[$m = 3k+2$  (here, $k = 2$, $g = 4$)]{
\begin{tikzpicture}[vtx/.style={draw, circle, inner sep = 1.5 pt, font = \tiny}]
\pgfmathsetmacro{\r}{1.5}
\foreach \i in {0,...,5}{
\node[vtx] (v\i) at ($(360*\i/6: \r)+(0,-.6)$){};
}
\foreach \i in {0,1,2,4}{
\node[red, vtx] (w\i) at ($(360*\i/6 + 30 : \r)+(0,1.5)$){};
\node[blue, vtx] (u\i) at ($(360*\i/6 + 30: \r-.5)+(0,1.5)$){};
}

\foreach \j/\k/\bend in {w/red/22, u/blue/-20}{
\draw[\k] (v1) -- (\j0);
\draw[\k] (\j0) -- (\j1);
\draw[\k] (\j1) -- (\j2);
\draw[\k] (\j2) -- (v2);
\draw[\k] (v2) -- (\j4);
\draw[\k] (\j4) -- (v1);
\draw[\k] (\j1) to[bend left =\bend](\j4);
\draw[\k] (\j2) to[bend left = \bend](v1);
\draw[\k] (\j0) to[bend left = \bend](v2);
}

\draw[thick] (v1) -- (v0) -- (v5) -- (v4) -- (v3) -- (v2);
\draw [thick](v2) to [bend left = -20] (v5);
\draw[thick] (v1) to [bend left = 20] (v4);
\draw[thick] (v0) -- (v3);

\node[vtx, fill = green]at (v2){};
\node[vtx, fill = green]at (v1){};

\end{tikzpicture}
}
\hspace{1cm}
\subfloat[$m = 3k+1$]{
\begin{tikzpicture}[vtx/.style={draw, circle, inner sep = 1.5 pt, font = \tiny}]
\pgfmathsetmacro{\r}{1.5}
\foreach \i in {0,...,5}{
\node[vtx] (v\i) at ($(360*\i/6: \r)+(0,-1)$){};
}
\foreach \i in {0,1,2,4}{
\node[red, vtx] (w\i) at ($(360*\i/6 + 30 : \r)+(0,1.5)$){};
\node[blue, vtx] (u\i) at ($(360*\i/6 + 30: \r-.5)+(0,1.5)$){};
}

\node[vtx, fill = green] (x) at ($(v1) + (0,.4)$){};
\node[vtx, fill = green] (y) at ($(v2) + (0,.4)$){};

\foreach \j/\k/\bend in {w/red/22, u/blue/-20}{
\draw[\k] (x) -- (\j0);
\draw[\k] (\j0) -- (\j1);
\draw[\k] (\j1) -- (\j2);
\draw[\k] (\j2) -- (y);
\draw[\k] (y) -- (\j4);
\draw[\k] (\j4) -- (x);
\draw[\k] (\j1) to[bend left =\bend](\j4);
\draw[\k] (\j2) to[bend left = \bend](x);
\draw[\k] (\j0) to[bend left = \bend](y);
}

\draw[thick] (v1) -- (v0) -- (v5) -- (v4) -- (v3) -- (v2);
\draw[thick] (v2) to [bend left = -20] (v5);
\draw[thick] (v1) to [bend left = 20] (v4);
\draw[thick] (v0) -- (v3);
\draw[thick] (v1) -- (x);
\draw[thick](v2) -- (y);


\end{tikzpicture}
}
\caption{Illustrating the construction in Theorem \ref{thm:identify}, which produces semicubic $(3,m; g)$ graphs of girth $g$ beginning with input graphs with $n_{g}$ vertices. For the purposes of illustration, the construction is shown using $K_{3,3}$, which is the unique $(3,4)$-cage; (a) is a $(3,9; 4)$ graph with $3(6-2)+2 = 14$ vertices, (b) is a $(3, 8; 4)$ graph  with $2(6-2) + 6 = 14$ vertices; (c) is a $(3,7; 4)$ graph with $2(6 - 2) + 6+ 2 = 16$ vertices.  }
\label{fig:identify}
\end{center}
\end{figure}

Taking into account that the order of each of the $(3;10)$-cages is equal to $70$ (recall that the girth 10 cages are the Balaban Cage and two others $\cite{PBMOG04}$) and it is easy to find two remote vertices at distance $5$ in the Balaban cage, for example, we obtain the following corollary: 

\begin{corollary}\label{girth10}
There exist $(3,m;10)$-graphs of order: 
\begin{itemize}
\item $22m+\frac{2m}{3}+2$ for $m=3k$
\item $22m+\frac{2m}{3}+49+\frac{1}{3}$ for $m=3k+1$
\item $22m+\frac{2m}{3}+24+\frac{2}{3}$ for $m=3k+2$
\end{itemize}
\end{corollary}

And also, taking into account that the Moore $(3;12)$-cage has order $126$ and it is the incidence graph of the generalized hexagon of order $2$, which also has two vertices at distance $6$, it follows that: 

\begin{corollary}\label{girth12}
There exist $(3,m;12)$-graphs of order: 
\begin{itemize}
\item $41m+\frac{m}{3}+2$ for $m=3k$
\item $41m+\frac{m}{3}+86+\frac{2}{3}$ for $m=3k+1$
\item $41m+\frac{m}{3}+43+\frac{1}{3}$ for $m=3k+2$
\end{itemize}
\end{corollary} 

Finally, we would like to calculate these graphs for girth $14$ using the smallest $(3;14)$-graph known currently (the \emph{record} graph), which was given by Exoo in \cite{E09} of order $348$. The current lower bound for a $(3;14)$-graph is 258, given by Mc Kay et. al. in \cite{MMP00}. From this construction, it follows that:

\begin{corollary}\label{girth14}
There exist $(3,m;14)$-graphs of order: 
\begin{itemize}
\item $115m+\frac{m}{3}+2$ for $m=3k$
\item $115m+\frac{m}{3}+234+\frac{2}{3}$ for $m=3k+1$
\item $115m+\frac{m}{3}+117+\frac{1}{3}$ for $m=3k+2$
\end{itemize}
\end{corollary}

\section{Preliminaries on voltage graphs and derived graphs} \label{prel}
In this section, we present definitions and preliminary results about voltage graphs, and we exhibit some voltage and derived graphs constructed with them. 


Following standard references (e.g., \cite{BP03,GT87,PS10}), a voltage graph $G$ is a labeled directed multigraph, often including loops and parallel edges, along with a group $\Gamma$; the labels on the edges are elements of $\Gamma$. Throughout this paper, $\Gamma$ is a cyclic group $\mathbb{Z}_{m}$ with addition as the group operation. The \emph{derived} graph $G_{m}$, also called the \emph{lift} graph, for a voltage graph with voltage group $\mathbb{Z}_{m}$ is formed from $G$ as follows: each vertex $v$ in $G$ corresponds to $m$ vertices in $G_{m}$, labelled $v^{0}, \cdots, v^{m-1}$. An arrow in $G$ from $v$ to $w$ labelled $a$ means that vertex $v^{i}$ and vertex $w^{i+a}$ are connected by an edge in $G_{m}$, with all indices throughout the paper taken modulo $m$.\footnote{Note that often, elements of a single orbit of elements in the lift graph are labeled with subscripts, but here, we are using superscripts to label the elements in the orbit, and reserving subscripts to label vertices in the voltage graph. See Figure \ref{fig:pinned vertex} for an example of this indexing.} Note that we could also have drawn an arrow from $w$ to $v$ labeled $-a$ and produced the same edges in the lift. If vertex $v$ is incident with a loop labeled $a$ in $G$, then in $G_{m}$, vertices $v^{i}$ and $v^{i+a}$ are incident. Figure \ref{fig:voltEx} shows the standard drawing of the Heawood graph, the $(3,6)$-cage, as a $\mathbb{Z}_{7}$ lift of a voltage graph on two vertices. 

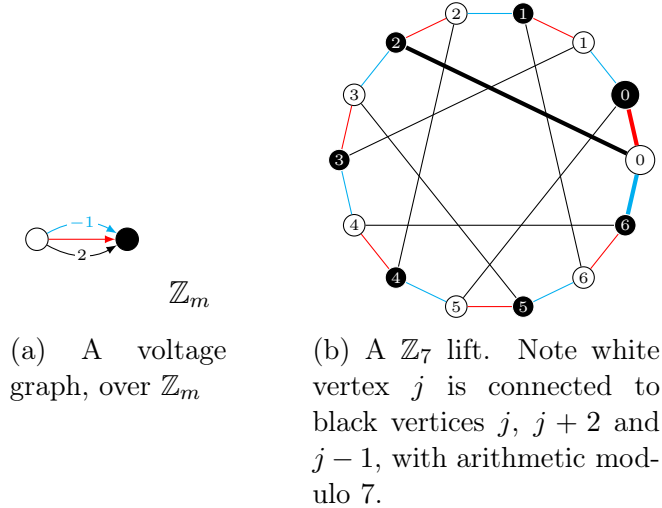
\begin{figure}[htbp]
\begin{center}
\subfloat[
A voltage graph, over $\mathbb{Z}_{m}$]{
\begin{tikzpicture}[lbl/.style = {midway, fill = white, inner sep = 1 pt, font=\tiny}, scale = .6, ]
\node[draw, circle, inner sep = 3 pt] (A) at (0,0){};
\node[draw, circle, inner sep = 3 pt, fill = black] (B) at (2,0){};
\draw[-latex, red] (A) to[bend left = 0]  (B);
\draw[-latex, cyan] (A) to[bend left = 30] node[lbl]{$-1$} (B);
\draw[-latex, black] (A) to[bend right = 30] node[midway, fill = white, inner sep = 1 pt, font=\tiny]{$2$} (B);
\node[below right= 12 pt of B] {$\mathbb{Z}_{m}$};
\end{tikzpicture}
}
\hspace{1 cm}
\subfloat[A $\mathbb{Z}_{7}$ lift. Note white vertex $j$ is connected to black vertices $j$, $j+2$ and $j-1$, with arithmetic modulo 7.
]{
\begin{tikzpicture}
\node[draw, circle,  inner sep = 2 pt, font=\tiny] (v0) at (360*0/7:2) {0};
\node[draw, circle, white, fill=black, inner sep = 2 pt, font=\tiny] (w0) at (360*0/7+ 180/7:2) {0};

\foreach \j in {1,...,6}{
\node[draw, circle,  inner sep = 1 pt, font=\tiny] (v\j) at (360*\j/7:2) {\j};
\node[draw, circle, white, fill = black, inner sep = 1 pt, font=\tiny] (w\j) at (360*\j/7 + 180/7:2) {\j};
}

\draw[ultra thick,red] (v0) -- (w0);
\foreach \j in {1,...,6}{
\draw[red] (v\j) -- (w\j);
}

\draw[ultra thick,cyan] (v0) -- (w6);
\foreach \j in {1,...,6}{
\draw[cyan] let \n1 = {int(mod(\j-1,7))} in (v\j) -- (w\n1);
}

\draw[ultra thick,black] (v0) -- (w2);
\foreach \j in {1,...,6}{
\draw[black] let \n1 = {int(mod(\j+2,7))} in (v\j) -- (w\n1);
}

\end{tikzpicture}

}
\caption{The Heawood graph, the $(3,6)$-cage, may be drawn with $7$-fold rotational symmetry as on the right. It  can be drawn naturally as a $\mathbb{Z}_{7}$ lift of the voltage graph shown to the left; $\mathbb{Z}_{7}$-orbits are indicated with color, and the 0th element of each symmetry class is shown larger/thicker.
}
\label{fig:voltEx}
\end{center}
\end{figure}

The collection of directed edges and loops in a voltage graph, along with their labels, is called a \emph{voltage assignment}. Similarly, if $S$ is a subgraph of $G$, the voltage assignment $v(S)$ is the set of edges and labels in $S$. Throughout the paper, an unlabeled edge in a voltage graph is assumed to have voltage assignment 0, which we also draw as undirected.

A final modification of the voltage graph construction, which as far as we know is new to this paper, is the notion of a \emph{pinned vertex}, which is a special vertex of degree 1 in the voltage graph. We indicate this construction in the voltage graph using the symbol $*$ (also indicated with $\Box$ in figures). 
Specifically, the notation \tikz{\node[draw, inner sep = 2 pt] (v) at (0,0){$v^{*}$}; \node[draw, circle, inner sep = 2 pt] (w) at (1,0) {$w$}; \draw[](v) -- (w) ;} over $\mathbb{Z}_{m}$, where vertex $v^{*}$ is a pinned vertex, indicates that in the lift graph, there is a single vertex labeled $v^{*}$, which is connected to each of the vertices $w^{i}$, $i = 0, \ldots, m-1$. It follows that in a derived graph, when the voltage group is $\mathbb{Z}_{m}$, 
 the degree of a pinned vertex is $m$. 
Figure \ref{fig:pinned vertex} shows an example of a voltage graph with two pinned vertices and the corresponding derived graph over $\mathbb{Z}_{3}$.

\begin{figure}[htbp]
\begin{center}
\subcaptionbox{A voltage graph with two pinned vertices}[.3\linewidth]{
\begin{tikzpicture}[vtx/.style={draw, circle, inner sep = 1 pt, font = \scriptsize},lbl/.style = {midway, fill = white, inner sep = 2 pt, font=\small}, scale = .8, every node/.append style={transform shape}]]
\node[draw] (yy) at (0,0) {$y^{*}$};
\node[vtx, white,fill = black] (y) at (0,1) {$y$};
\node[vtx, white,fill = cyan] (y0) at (1,1.75) {$y_{0}$};
\node[vtx, white, fill = red] (x0) at (1,3.25) {$x_{0}$};
\node[vtx, fill = green] (x) at (0,4) {$x$};
\node[draw, fill = gray!50] (xx) at (0,5) {$x^{*}$};
\draw[] (x) to[bend right = 20] (y);
\draw[dashed] (xx) to (x);
\draw[dotted, ultra thick] (yy) to (y);
\draw[] (y) to (y0);
\draw[] (x) to (x0);
\draw[-latex, red] (x0) to[bend left = 30] node[lbl]{1} (y0);
\draw[-latex, cyan] (x0) to[bend right = 30] node[lbl]{2} (y0);
\path node[below right = 6 pt and 12 pt of y] {$\mathbb{Z}_{m}$};
\end{tikzpicture}
}
\hfill
\subcaptionbox{The $\mathbb{Z}_{3}$ lift...}[.3\linewidth]{
\begin{tikzpicture}[vtx/.style={draw, circle, inner sep = 1 pt, font = \scriptsize},lbl/.style = {midway, fill = white, inner sep = 2 pt, font=\scriptsize}, scale = .8, every node/.append style={transform shape}]
\def\r{2}
\node[draw] (yy) at (\r*1,0) {$y^{*}$};
\node[draw, fill = gray!50] (xx) at (\r*1,5) {$x^{*}$};
\foreach \j in {0,1,2}{
\node[vtx, white, fill = black, inner sep= 1 pt] (y\j) at (0+\r*\j,1) {$y^{\j}$};
\node[vtx, white, fill = cyan,inner sep= 1 pt] (y0\j) at (1+\r*\j,1.75) {$y_{0}^{\j}$};
\node[vtx, white, fill = red,inner sep= 1 pt] (x0\j) at (1+\r*\j,3.25) {$x_{0}^{\j}$};
\node[vtx, fill = green,inner sep= 1 pt] (x\j) at (0+\r*\j,4) {$x^{\j}$};
}

%
\foreach \j in {0,1,2}{
\draw[dashed, thick] (xx) to (x\j);
\draw[dotted, ultra thick] (yy) to (y\j);
\draw[thick] (x\j) to[bend right = 10]   (y\j);
\draw[thick] (y\j) to (y0\j);
\draw[thick] (x\j) to (x0\j);
\draw[thick,red] let \n1 = {int(mod(\j+1,3))} in (x0\j) to[bend left = 0] (y0\n1);
\draw[thick,cyan] let \n1 = {int(mod(\j+2,3))} in (x0\j) to[bend right = 0] (y0\n1);
}
\end{tikzpicture}
}
\hfill
\subcaptionbox{...is isomorphic to the Heawood graph.}[.3\linewidth]{
\begin{tikzpicture}[vtx/.style={draw, circle, inner sep = 1 pt, font=\scriptsize}, scale = .8, every node/.append style={transform shape}]
\def\r{2.2}
\node[vtx, white, fill=red] (x00) at (360*0/14:\r){$x_{0}^{0}$};
\node[vtx, white,fill=cyan] (y02) at (360*1/14:\r){$y_{0}^{2}$};
\node[vtx, white, fill=black] (y2) at (360*2/14:\r){$y^{2}$};
\node[vtx, fill = green] (x2) at (360*3/14:\r){$x^{2}$};
\node[draw, fill = gray!50] (xx) at (360*4/14:\r){$x^{*}$};
\node[vtx, fill = green] (x0) at (360*5/14:\r){$x^{0}$};
\node[vtx, white, fill = black] (y0) at (360*6/14:\r){$y^{0}$};
\node[draw] (yy) at (360*7/14:\r){$y^{*}$};
\node[vtx, white, fill = black] (y1) at (360*8/14:\r){$y^{1}$};
\node[vtx, fill = green] (x1) at (360*9/14:\r){$x^{1}$};
\node[vtx, white, fill = red] (x01) at (360*10/14:\r){$x_{0}^{1}$};
\node[vtx, white,fill = cyan] (y00) at (360*11/14:\r){$y_{0}^{0}$};
\node[vtx, white,fill = red] (x02) at (360*12/14:\r){$x_{0}^{2}$};
\node[vtx,white, fill = cyan] (y01) at (360*13/14:\r){$y_{0}^{1}$};
\foreach \j in {0,1,2}{
\draw[dashed, thick] (xx) to (x\j);
\draw[dotted, ultra thick] (yy) to (y\j);
\draw[thick] (x\j) to[bend right = 00]   (y\j);
\draw[thick] (y\j) to (y0\j);
\draw[thick] (x\j) to (x0\j);
\draw[thick,red] let \n1 = {int(mod(\j+1,3))} in (x0\j) to[bend left = 0] (y0\n1);
\draw[thick,cyan] let \n1 = {int(mod(\j+2,3))} in (x0\j) to[bend right = 0] (y0\n1);
}
\end{tikzpicture}
}
\caption{An example of a voltage graph with pinned vertices. In this and subsequent figures,  each unpinned vertex $v$ in the voltage graph becomes the collection of vertices $v^{0}, v^{1}, \ldots, v^{m-1}$ in the lift, whereas pinned vertices $x^{*}$ and $y^{*}$ lift to single vertices.}
\label{fig:pinned vertex}
\end{center}
\end{figure}
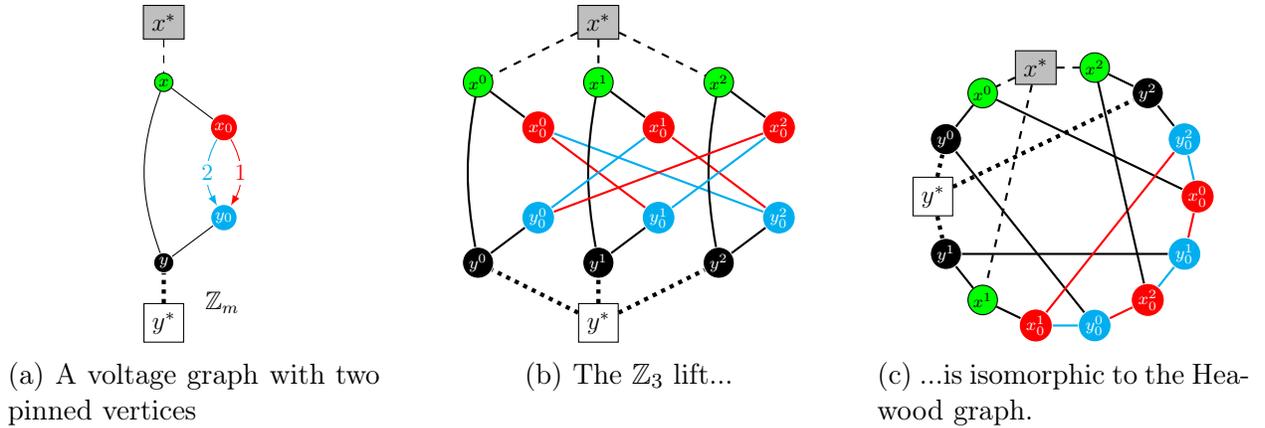

\begin{remark} Given a voltage graph $G$ with voltage group $\mathbb{Z}_{m}$ with a collection of pinned vertices $v^{*}_{1}, \ldots v^{*}_{s}$ in which all vertices are degree $r$ except the pinned vertices, which are of degree 1, the derived graph $G_{m}$ is a $(r, m)$-biregular graph with $s$ vertices of degree $m$.\end{remark}

In what follows, we analyze the girth of graphs derived from voltage graphs with pinned vertices but no loops. We begin with the following well-known result (see, e.g., \cite[\S 2.1.3 - 4]{GT87}).  A \emph{non-reversing closed walk} in a voltage graph is a closed walk in which no edge is repeated consecutively, although the same edge or vertex might be traversed more than once later in the walk. Note that constructing the (directed) walk in the voltage graph may require reversing some of the displayed arrows in the voltage graph (and consequently negating the labels) in order to have the walk have a consistent direction. 
For example, the red walk in Figure \ref{fig:4cyclesAtVtx} repeats a vertex, and the walk in Figure \ref{fig:4cyclesJoinEdge} repeats an edge, but not consecutively.

\begin{lemma}\label{nonreversing} If the sum of the labels in a  non-reversing closed walk $W$ in a voltage graph with voltage group $\mathbb{Z}_{m}$ is congruent to $0 \bmod m$, and no smaller sub-walk  has voltage sum congruent to $0 \bmod m$, then the lift of $W$ forms a cycle.\end{lemma}

\begin{proof} This follows from \cite[\S 2.1.3 - 4]{GT87}. It is clear that the lift is a closed walk, since the starting and ending vertices in $G_{m}$ have the same index. To see that it is a cycle, suppose we begin the walk $W$ at some vertex $v$, and suppose the lift walk had also intersected itself at some vertex $w^{j}$ (in $G_{m}$). That would mean that in the voltage graph, the vertex $w$ was used twice in the voltage walk, and moreover, the voltages along the part of the walk between those encounters with $w$ summed to 0. This contradicts the assumption that no sub-walk had net voltage 0.   
\end{proof}


Of particular utility to us is that if the sum  $s$ of the labels along a cycle in the voltage graph divides $m$, so that $m = s d$, then the closed walk formed by going $d$ times around the cycle in the voltage graph lifts to a cycle in the derived graph. However, there are other closed walks in the voltage graph that can also lift to short cycles, which we discuss in subsection \ref{sec:closedWalks}.

When we are considering voltage graphs with pinned vertices, however, there are additional cycles in the lift graph which pass through the pinned vertices that do not come from lifting closed non-reversing walks in the voltage graph whose voltages sum to 0.
See Figure \ref{fig:cycleLifts} 
for a collection of specific examples, and Figures \ref{fig:genericPath} and \ref{fig:lollipop} for general schematics.

\begin{figure}[htbp]
\begin{center}

\subcaptionbox{A 4-cycle in a voltage graph whose labels sum to 1. }[.23\linewidth]{
\begin{tikzpicture}[lbl/.style = {midway, fill = white, inner sep = 2 pt, font=\footnotesize}, scale = .8, every node/.append style={transform shape}]
\node[draw] (yy) at (0,0) {};
\node[draw, circle, fill = black] (y) at (0,1) {};
\node[draw, circle, fill = cyan] (y0) at (1,1.75) {};
\node[draw, circle, fill = red] (x0) at (1,3.25) {};
\node[draw, circle, fill = green] (x) at (0,4) {};
\node[draw, fill = gray!50] (xx) at (0,5) {};
\draw[] (x) to[bend right = 20] (y);
\draw[dashed] (xx) to (x);
\draw[dashed] (yy) to (y);
\draw[] (y) to (y0);
\draw[] (x) to (x0);
\draw[-latex, red] (x0) to[bend left = 30] node[lbl]{1} (y0);
\draw[-latex, cyan] (x0) to[bend right = 30] node[lbl]{2} (y0);
\path node[below right = 6 pt and 12 pt of y] {$\mathbb{Z}_{m}$};

\begin{scope}[on background layer]
\draw[line width = 5 pt, opacity = .5, yellow] (y) to[bend left = 20] (x.center) -- (x0) to[bend left = 30] (y0.center) -- (y);
\end{scope}
\end{tikzpicture}
}
\hfill
\subcaptionbox{A lollipop 4-cycle involving a pinned vertex.}[.23\linewidth]{
\begin{tikzpicture}[lbl/.style = {midway, fill = white, inner sep = 2 pt, font=\footnotesize}, scale = .8, every node/.append style={transform shape}]
\node[draw] (yy) at (0,0) {};
\node[draw, circle, fill = black] (y) at (0,1) {};
\node[draw, circle, fill = cyan] (y0) at (1,1.75) {};
\node[draw, circle, fill = red] (x0) at (1,3.25) {};
\node[draw, circle, fill = green] (x) at (0,4) {};
\node[draw, fill = gray!50] (xx) at (0,5) {};
\draw[] (x) to[bend right = 20] (y);
\draw[dashed] (xx) to (x);
\draw[dashed] (yy) to (y);
\draw[] (y) to (y0);
\draw[] (x) to (x0);
\draw[-latex, red] (x0) to[bend left = 30] node[lbl]{1} (y0);
\draw[-latex, cyan] (x0) to[bend right = 30] node[lbl]{2} (y0);
\path node[below right = 6 pt and 12 pt of y] {$\mathbb{Z}_{m}$};

\begin{scope}[on background layer]
\draw[line width = 5 pt, opacity = .5, yellow] (y) to[bend left = 20] (x.center) -- (x0) to[bend left = 30] (y0.center) -- (y);
\draw[line width = 5 pt, opacity = .5, yellow] (x) -- (xx);
\end{scope}
\end{tikzpicture}
}
\hfill
\subcaptionbox{A lollipop walk with a path of length 2 and a cycle of length 2.}[.23\linewidth]{
\begin{tikzpicture}[vtx/.style={draw, circle, inner sep = 1 pt, font = \tiny},
lbl/.style = {midway, fill = white, inner sep = 2 pt, font=\footnotesize}, scale = .8, every node/.append style={transform shape}]
\node[draw] (yy) at (0,0) {};
\node[draw, circle, fill = black] (y) at (0,1) {};
\node[draw, circle, fill = cyan] (y0) at (1,1.75) {};
\node[draw, circle, fill = red] (x0) at (1,3.25) {};
\node[draw, circle, fill = green] (x) at (0,4) {};
\node[draw, fill = gray!50] (xx) at (0,5) {};
\draw[] (x) to[bend right = 20] (y);
\draw[dashed] (xx) to (x);
\draw[dashed] (yy) to (y);
\draw[] (y) to (y0);
\draw[] (x) to (x0);
\draw[-latex, red] (x0) to[bend left = 30] node[lbl]{1} (y0);
\draw[-latex, cyan] (x0) to[bend right = 30] node[lbl]{2} (y0);
\path node[below right = 6 pt and 12 pt of y] {$\mathbb{Z}_{m}$};


\begin{scope}[on background layer]
\draw[line width = 5 pt, opacity = .5, yellow] (xx) -- (x) -- (x0) to[bend left = 30] (y0.center) to[bend left = 30] (x0.center) -- (x) -- (xx);

\end{scope}
\end{tikzpicture}
}
\hfill
\subcaptionbox{A path of length 3 between two pinned vertices.}
[.2\linewidth]{
\begin{tikzpicture}[vtx/.style={draw, circle, inner sep = 1 pt, font = \tiny},lbl/.style = {midway, fill = white, inner sep = 2 pt, font=\footnotesize}, scale = .8, every node/.append style={transform shape}]
\node[draw] (yy) at (0,0) {};
\node[draw, circle, fill = black] (y) at (0,1) {};
\node[draw, circle, fill = cyan] (y0) at (1,1.75) {};
\node[draw, circle, fill = red] (x0) at (1,3.25) {};
\node[draw, circle, fill = green] (x) at (0,4) {};
\node[draw, fill = gray!50] (xx) at (0,5) {};
\draw[] (x) to[bend right = 20] (y);
\draw[dashed] (xx) to (x);
\draw[dashed] (yy) to (y);
\draw[] (y) to (y0);
\draw[] (x) to (x0);
\draw[-latex, red] (x0) to[bend left = 30] node[lbl]{1} (y0);
\draw[-latex, cyan] (x0) to[bend right = 30] node[lbl]{2} (y0);
\path node[below right = 6 pt and 12 pt of y] {$\mathbb{Z}_{m}$};

\begin{scope}[on background layer]
\draw[line width = 5 pt, opacity = .5, yellow] (xx) -- (x) to[bend right = 20] (y.center) -- (yy);
\end{scope}
\end{tikzpicture}
}

\subcaptionbox{Three times around the 4-cycle lifts to  a 12-cycle.}[.23\linewidth]{
\begin{tikzpicture}[lbl/.style = {midway, fill = white, inner sep = 2 pt, font=\tiny}, scale = .8, every node/.append style={transform shape, font = \scriptsize}]
\def\r{1.5}
\node[draw] (yy) at (\r*1,0) {};
\node[draw, fill = gray!50] (xx) at (\r*1,5) {};
\foreach \j in {0,1,2}{
\node[draw, circle, white, fill = black, inner sep= 1 pt] (y\j) at (0+\r*\j,1) {\j};
\node[draw, circle, fill = cyan,inner sep= 1 pt] (y0\j) at (1+\r*\j,1.75) {\j};
\node[draw, circle, white, fill = red,inner sep= 1 pt] (x0\j) at (1+\r*\j,3.25) {\j};
\node[draw, circle, fill = green,inner sep= 1 pt] (x\j) at (0+\r*\j,4) {\j};
}

%
\foreach \j in {0,1,2}{
\draw[dashed, thick] (xx) to (x\j);
\draw[dashed, thick] (yy) to (y\j);
\draw[thick] (x\j) to[bend right = 10]   (y\j);
\draw[thick] (y\j) to (y0\j);
\draw[thick] (x\j) to (x0\j);
\draw[thick,red] let \n1 = {int(mod(\j+1,3))} in (x0\j) to[bend left = 0] (y0\n1);
\draw[thick,cyan] let \n1 = {int(mod(\j+2,3))} in (x0\j) to[bend right = 0] (y0\n1);
}

\begin{scope}[on background layer]
\foreach \j in {0,1,2}{
\draw[line width = 5 pt, opacity = .5, yellow] (y\j) to[bend left = 10] (x\j.center);
\draw[line width = 5 pt, opacity = .5, yellow] (x\j) -- (x0\j);
\draw[line width = 5 pt, opacity = .5, yellow] let \n1 = {int(mod(\j+1,3))} in  (x0\j) -- (y0\n1);
\draw[line width = 5 pt, opacity = .5, yellow] (y0\j) -- (y\j);
}

\end{scope}
\end{tikzpicture}
}
%
%
\hfill
\subcaptionbox{The lift of the lollipop 4-cycle.}[.23\linewidth]{
\begin{tikzpicture}[lbl/.style = {midway, fill = white, inner sep = 2 pt, font=\tiny}, scale = .8, every node/.append style={transform shape, font = \scriptsize}]
\def\r{1.5}
\node[draw] (yy) at (\r*1,0) {};
\node[draw, fill = gray!50] (xx) at (\r*1,5) {};
\foreach \j in {0,1,2}{
\node[draw, circle, white, fill = black, inner sep= 1 pt] (y\j) at (0+\r*\j,1) {\j};
\node[draw, circle, fill = cyan,inner sep= 1 pt] (y0\j) at (1+\r*\j,1.75) {\j};
\node[draw, circle, white, fill = red,inner sep= 1 pt] (x0\j) at (1+\r*\j,3.25) {\j};
\node[draw, circle, fill = green,inner sep= 1 pt] (x\j) at (0+\r*\j,4) {\j};
}

%
\foreach \j in {0,1,2}{
\draw[dashed, thick] (xx) to (x\j);
\draw[dashed, thick] (yy) to (y\j);
\draw[thick] (x\j) to[bend right = 10]   (y\j);
\draw[thick] (y\j) to (y0\j);
\draw[thick] (x\j) to (x0\j);
\draw[thick,red] let \n1 = {int(mod(\j+1,3))} in (x0\j) to[bend left = 0] (y0\n1);
\draw[thick,cyan] let \n1 = {int(mod(\j+2,3))} in (x0\j) to[bend right = 0] (y0\n1);
}

\begin{scope}[on background layer]

\draw[line width = 5 pt, opacity = .5, yellow] (xx) -- (x0) -- (x00) -- (y01) -- (y1) to[bend left = 10] (x1.center) -- (xx);
\end{scope}

\end{tikzpicture}
}
\hfill
\subcaptionbox{The lift of the lollipop with a path of length 2 and a 2-cycle.}[.23\linewidth]{
\begin{tikzpicture}[lbl/.style = {midway, fill = white, inner sep = 2 pt, font=\tiny}, scale = .8, every node/.append style={transform shape, font = \scriptsize}]
\def\r{1.5}
\node[draw] (yy) at (\r*1,0) {};
\node[draw, fill = gray!50] (xx) at (\r*1,5) {};
\foreach \j in {0,1,2}{
\node[draw, circle, white, fill = black, inner sep= 1 pt] (y\j) at (0+\r*\j,1) {\j};
\node[draw, circle, fill = cyan,inner sep= 1 pt] (y0\j) at (1+\r*\j,1.75) {\j};
\node[draw, circle, white, fill = red,inner sep= 1 pt] (x0\j) at (1+\r*\j,3.25) {\j};
\node[draw, circle, fill = green,inner sep= 1 pt] (x\j) at (0+\r*\j,4) {\j};
}

%
\foreach \j in {0,1,2}{
\draw[dashed, thick] (xx) to (x\j);
\draw[dashed, thick] (yy) to (y\j);
\draw[thick] (x\j) to[bend right = 10]   (y\j);
\draw[thick] (y\j) to (y0\j);
\draw[thick] (x\j) to (x0\j);
\draw[thick,red] let \n1 = {int(mod(\j+1,3))} in (x0\j) to[bend left = 0] (y0\n1);
\draw[thick,cyan] let \n1 = {int(mod(\j+2,3))} in (x0\j) to[bend right = 0] (y0\n1);
}

\begin{scope}[on background layer]

\draw[line width = 5 pt, opacity = .5, yellow] (xx) -- (x0) -- (x00) -- (y01) -- (x02) --(x2) -- (xx);
\end{scope}

\end{tikzpicture}
}
\hfill
\subcaptionbox{Lifting a path of length 3 between two pinned vertices.}[.23\linewidth]{
\begin{tikzpicture}[lbl/.style = {midway, fill = white, inner sep = 2 pt, font=\tiny}, scale = .8, every node/.append style={transform shape, font = \scriptsize}]
\def\r{1.5}
\node[draw] (yy) at (\r*1,0) {};
\node[draw, fill = gray!50] (xx) at (\r*1,5) {};
\foreach \j in {0,1,2}{
\node[draw, circle, white, fill = black, inner sep= 1 pt] (y\j) at (0+\r*\j,1) {\j};
\node[draw, circle, fill = cyan,inner sep= 1 pt] (y0\j) at (1+\r*\j,1.75) {\j};
\node[draw, circle, white, fill = red,inner sep= 1 pt] (x0\j) at (1+\r*\j,3.25) {\j};
\node[draw, circle, fill = green,inner sep= 1 pt] (x\j) at (0+\r*\j,4) {\j};
}

%
\foreach \j in {0,1,2}{
\draw[dashed, thick] (xx) to (x\j);
\draw[dashed, thick] (yy) to (y\j);
\draw[thick] (x\j) to[bend right = 10]   (y\j);
\draw[thick] (y\j) to (y0\j);
\draw[thick] (x\j) to (x0\j);
\draw[thick,red] let \n1 = {int(mod(\j+1,3))} in (x0\j) to[bend left = 0] (y0\n1);
\draw[thick,cyan] let \n1 = {int(mod(\j+2,3))} in (x0\j) to[bend right = 0] (y0\n1);
}

\begin{scope}[on background layer]

\draw[line width = 5 pt, opacity = .5, yellow] (xx) -- (x0) to[bend right = 10] (y0.center) -- (yy) -- (y1) to[bend left = 10] (x1.center) -- (xx);

\end{scope}

\end{tikzpicture}
}

\caption{Examples of various walks in a voltage graph and their $\mathbb{Z}_{3}$ lifts. Note that traveling forward by 2 is the same as traveling backwrds by 1 when $m = 3$.}
\label{fig:cycleLifts}
\end{center}
\end{figure}
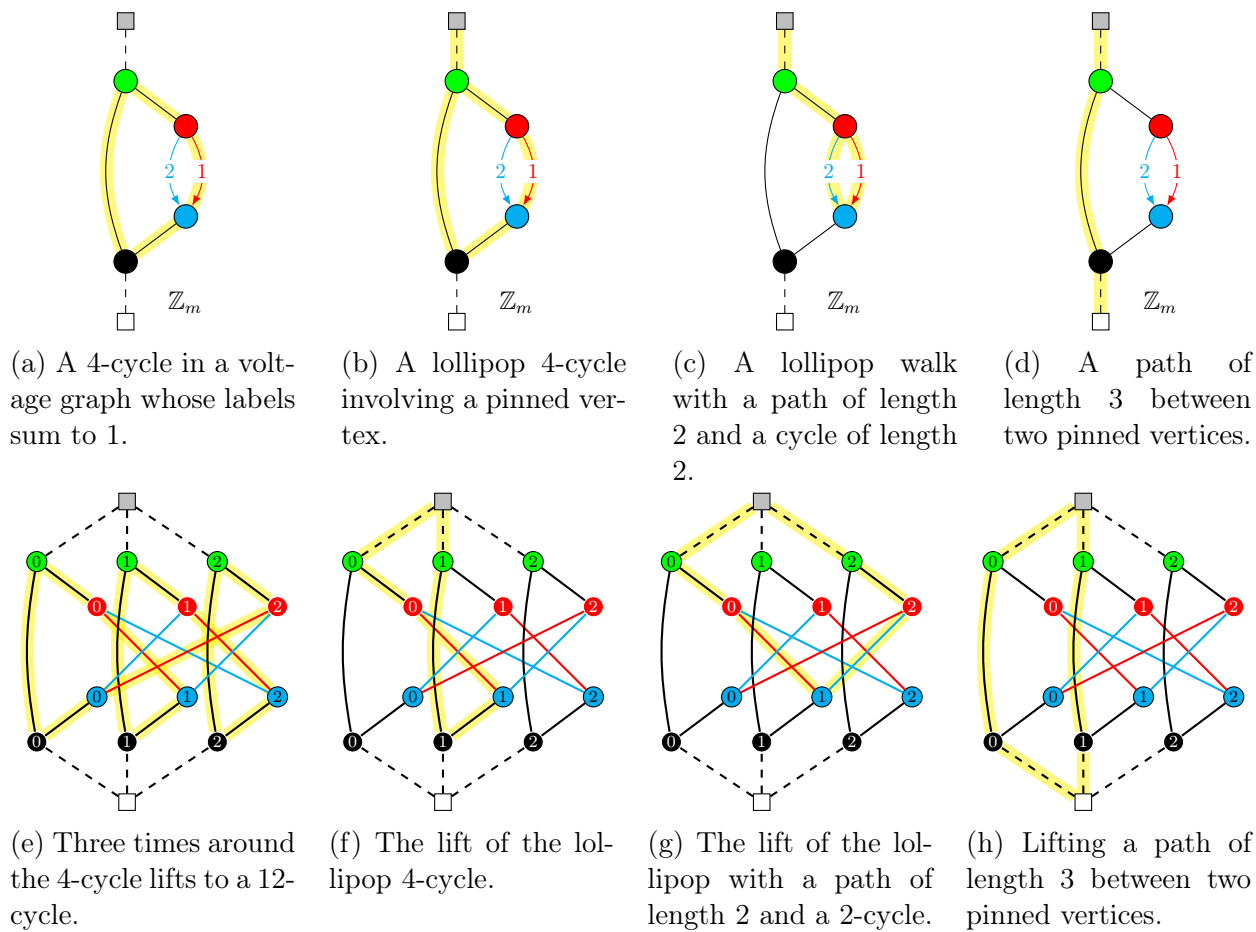

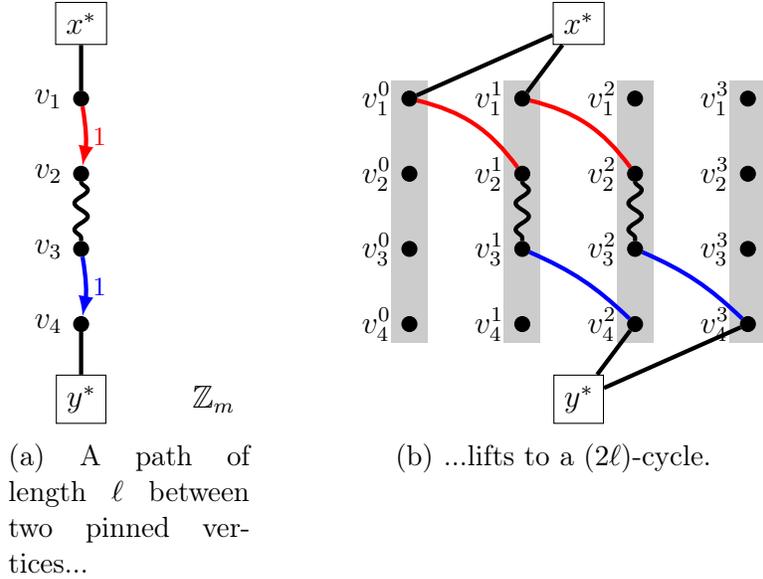
\begin{figure}[htbp]
\begin{center}
\subfloat[A path of length $\ell$ between two pinned vertices... ]{
\begin{tikzpicture}[lbl/.style = {midway, 
inner sep = 2 pt, font=\footnotesize}, vtx/.style = {draw, circle, fill = black, inner sep = 2 pt}, decoration = snake]
\node[draw,] (xx) at (0,6){$x^{*}$};
\node[vtx, 
label = left:$v_{1}$] (v1) at (0,5) {};
\node[vtx,
label = left:$v_{2}$] (v2) at (0,4){};
\node[vtx, 
label = left:$v_{3}$] (v3) at (0,3) {};
\node[vtx, 
label = left:$v_{4}$] (v4) at (0,2) {};
\node[draw, 
] (yy) at (0,1) {$y^{*}$};

\draw[,ultra thick] (xx) -- (v1);
\draw[-latex, ultra thick, red] (v1) to[bend left = 10]node[lbl, right]{$1$} (v2);
\draw[decorate,ultra thick] (v2) -- (v3);
\draw[-latex,  ultra thick, blue] (v3) to[bend left = 10]node[lbl, right]{$1$} (v4);
\draw[, ultra thick] (v4) -- (yy);

\node[right = of yy] {$\mathbb{Z}_{m}$};

\end{tikzpicture}
}
\hspace{1cm}
\subfloat[...lifts to a ($2 \ell$)-cycle.]{
\begin{tikzpicture}[lbl/.style = {midway, fill = white, inner sep = 2 pt, font=\footnotesize},vtx/.style = {draw, circle, fill = black, inner sep = 2 pt}, decoration = snake]
\def\r{1.5}
\node[draw] (yy) at (\r*3/2,1) {$y^{*}$};
\node[draw, ] (xx) at (\r*3/2,6) {$x^{*}$};
\foreach \j/\k in {0/0,1/1,2/2,3/3}{
\node[vtx,  label = left:$v_{1}^{\k}$] (v1\j) at (0+\r*\j,5) {};
\node[vtx,  label = left:$v_{2}^{\k}$] (v2\j) at (0+\r*\j,4) {};
\node[vtx,  label = left:$v_{3}^{\k}$] (v3\j) at (0+\r*\j,3) {};
\node[vtx,  label = left:$v_{4}^{\k}$] (v4\j) at (0+\r*\j,2) {};

}

%

\draw[, ultra thick] (xx) -- (v10);
\draw[, ultra thick] (xx) -- (v11);
\draw[, red,  ultra thick] (v10) to[bend left = 20]
(v21);
\draw[, red,  ultra thick] (v11) to[bend left = 20]
(v22);
\draw[decorate,ultra thick] (v21) -- (v31);
\draw[decorate,ultra thick] (v22) -- (v32);
\draw[, blue,  ultra thick] (v31) to[bend left = 10]
(v42);
\draw[, blue,  ultra thick] (v32) to[bend left = 10]
(v43);
\draw[,ultra thick] (v42) -- (yy);
\draw[,ultra thick] (v43) -- (yy);

\begin{scope}[on background layer]
\foreach \j in {0,1,2,3}
\node [fill=black!20,fit=(v1\j)(v2\j)(v3\j)(v4\j)]{};
\end{scope}
\end{tikzpicture}
}

\caption{A schematic illustrating that a path in $G$ of length $\ell$ between pinned vertices lifts to a cycle in $G_{m}$ of length $2 \ell$.}
\label{fig:genericPath}
\end{center}
\end{figure}

\begin{lemma}\label{lemma:pinnedwalks} Let $P$ be a path of length $\ell$ in a voltage graph $G$ that begins at a pinned vertex $x^{*}$ and ends at a (different) pinned vertex $y^{*}$. 
Then the lift graph contains a cycle of length $2 \ell$ passing through $x^{*}$ and $y^{*}$.  
\end{lemma}

\begin{proof} 
Assume that in $P$, $x^{*}$ is adjacent to $x$ and $y^{*}$ is adjacent to $y$, and let $P'$ be the subpath of $P$ in $G$ that starts at $x$ and ends at $y$. Construct two lifts of path $P'$: let $(P')^{0}$  start at $x^{0}$ and end at some $y^{j}$; then $(P')^{1}$ will start at $x^{1}$ and end at $y^{j+1}$. By construction, these paths are disjoint. Extending these paths back to $x^{*}$ and $y^{*}$ and joining them results in a cycle of length $2\ell$. (See Figure \ref{fig:genericPath}.)

\end{proof}

\begin{lemma}\label{lollipop} If $H$ is a ``lollipop'' subgraph in a voltage graph $G$ composed of a path of length $p$ from a pinned vertex $x^{*}$ to a non-pinned vertex $v$ and a cycle of length $q$ containing $v$ such that  the sum of the non-zero voltages along the cycle is not congruent to $0 \pmod m$ then $H$ lifts to a cycle of length $2p+q$ in the derived graph $G_{m}$. 
\end{lemma}

\begin{proof}

Suppose the sum of the voltages along the path $(x^{*}, x, \cdots, v)$ equals  $A$ and the sum of the voltages along the cycle equals $B$. In the lift, first travel from $x^{*}$ to $x^{0}$ to $v^{A}$ via the voltage instructions on the path in $G$. Next, travel along the lift of the cycle from $v^{A}$ to $v^{A+B}$, which is different from $v^{A}$ since $B \not\equiv 0 \bmod m$. By \cite[Theorem 2.1.2]{GT87} the lift of this cycle is a path in $G_{m}$. 
Finally,  travel backwards along the path from $v^{A+B}$ to $w^{B}$ (subtracting a total of $A$ voltages) and then to $x^{*}$. Since $B \not\equiv 0 \bmod m$, $x^{B}$ is different from $x^{0}$, so this closed walk does not self-intersect except at $x^{*}$ and thus is a cycle. (See Figure \ref{fig:lollipop}.)
\end{proof}

\begin{figure}[htbp]
\begin{center}
\subfloat[A lollipop with two arcs]{
\begin{tikzpicture}[lbl/.style = {midway, fill = white, inner sep = 2 pt, font=\footnotesize}, vtx/.style = {draw, circle, fill = black, inner sep = 2 pt}, decoration = snake]
\node[draw,] (xx) at (0,6){$x^{*}$};
\node[vtx, 
label = left:$v_{1}$] (v1) at (0,5) {};
\node[vtx,
label = left:$v_{2}$] (v2) at (0,4){};
\node[vtx, 
label = left:$v_{3}$] (v3) at (0,3) {};
\node[vtx, 
label = left:$v_{4}$] (v4) at (-1,2) {};
\node[vtx, label =  right:$v_{5}$
] (v5) at (1,1) {};

\draw[,ultra thick] (xx) -- (v1);
\draw[-latex, ultra thick, red] (v1) to[bend left = 10]node[lbl]{$a$} (v2);
\draw[decorate,ultra thick] (v2) -- (v3);
\draw[ultra thick, green] (v3) -- (v4);
\draw[-latex,  ultra thick, blue] (v4) to[bend right = 10]node[lbl]{$b$} (v5);
\draw[decorate, ultra thick, green] (v5) to[bend right = 50] (v3);

\node[right = of v5] {$\mathbb{Z}_{m}$};

\end{tikzpicture}
}
\hspace{1cm}
\subfloat[...lifts to a ($2p + q$)-cycle, as long as the sum of the arcs on the voltage graph cycle is not congruent to $0 \mod m$.]{
\begin{tikzpicture}[lbl/.style = {midway, fill = white, inner sep = 2 pt, font=\footnotesize},vtx/.style = {draw, circle, fill = black, inner sep = 2 pt}, decoration = snake]
\def\r{1.5}
\node[draw, ] (xx) at (\r*5/2,6) {$x^{*}$};
\foreach \j/\k in {0/0,1/a,2/{a+b}}{
\node[vtx,  label = left:$v_{1}^{\k}$] (v1\j) at (0+\r*\j,5) {};
\node[vtx,  label = left:$v_{2}^{\k}$] (v2\j) at (0+\r*\j,4) {};
\node[vtx,  label = left:$v_{3}^{\k}$] (v3\j) at (0+\r*\j,3) {};
\node[vtx,  label = left:$v_{4}^{\k}$] (v4\j) at (0+\r*\j,2) {};
\node[vtx,  label = left:$v_{5}^{\k}$] (v5\j) at (0+\r*\j,1) {};

}

\foreach \j/\k in {3/{b}}{
\node[vtx,  label = right:$v_{1}^{\k}$] (v1\j) at (0+\r*\j,5) {};
\node[vtx,  label = right:$v_{2}^{\k}$] (v2\j) at (0+\r*\j,4) {};
\node[vtx,  label = right:$v_{3}^{\k}$] (v3\j) at (0+\r*\j,3) {};
\node[vtx,  label = right:$v_{4}^{\k}$] (v4\j) at (0+\r*\j,2) {};
\node[vtx,  label = left:$v_{5}^{\k}$] (v5\j) at (0+\r*\j,1) {};
}
%

\draw[, ultra thick] (xx) -- (v10);
\draw[-latex, red,  ultra thick] (v10) to[bend left = 20]
(v21);
\draw[decorate,ultra thick] (v21) -- (v31);
\draw[ultra thick, green] (v31) -- (v41);
\draw[-latex, blue,  ultra thick] (v41) to[bend left = 20]
(v52);
\draw[decorate,ultra thick, green] (v52) to[bend right = 20] (v32);
\draw[decorate,ultra thick] (v32) -- (v22);
\draw[-latex, red,  ultra thick] (v22) to[bend right = 20]
(v13);
\draw[,ultra thick] (v13) -- (xx);

\begin{scope}[on background layer]
\foreach \j in {0,1,2,3}
\node [fill=black!20,fit=(v1\j)(v2\j)(v3\j)(v4\j)(v5\j)]{};
\end{scope}
\end{tikzpicture}
}

\caption{A schematic illustrating that a lollipop in $\mc G$ with path of length $p$  and cycle of length $q$ between pinned vertices lifts to a cycle of length $2 p + q$.}
\label{fig:lollipop}
\end{center}
\end{figure}
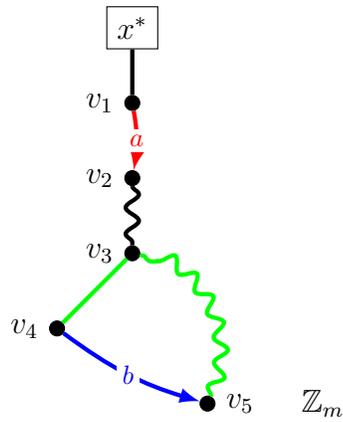
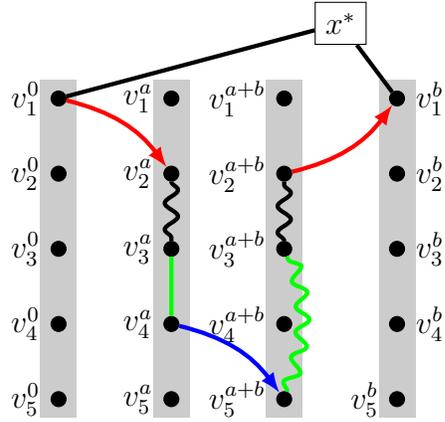

\subsection{Non-cycle closed walks that lift to short cycles\label{sec:closedWalks}}
In subsequent sections, we are interested in graphs of girth 6, 10, 8, and 12 that are formed by adding non-zero labeled arcs to the leaves of trees of particular heights. To analyze the girth, we need to analyze short closed walks in the voltage graph that lift to short cycles. In the previous section, we looked at cycles, lollipop walks, and paths between pinned vertices that lift to cycles. Here, we analyze possible non-cycle closed walks formed by joining cycles whose voltage sums we know. Some example walks are shown in Figure \ref{fig:joinCycles}; the context is that these are assumed to be parts of some voltage graph over $\mathbb{Z}_{m}$.

First, note that joining two short cycles at a vertex result in a short walk around both cycles whose length is the sum of the lengths of the cycles. Figure \ref{fig:4cyclesAtVtx} shows an example walk (following the red arrows) where the sums around the cycles are the same, caused by going forward around the first cycle (in the direction of the arcs in the cycle) and going backwards along the second cycle (opposite the direction of the arcs in the cycle).

Joining two cycles with a path leads to the construction of a non-reversing closed walk formed by traversing the first cycle, going along the path, going around the second cycle, and going back along the path to the starting point (Figure \ref{fig:4cyclesJoinEdge}). If we think about joining two 4-cycles along a shared edge, we can construct a closed walk by going down the shared edge, around the first cycle, back down the shared edge, and around the second cycle. This is a non-reversing walk of length 8 as well.

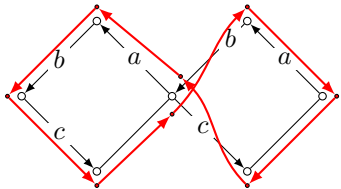
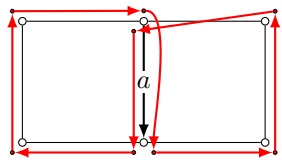
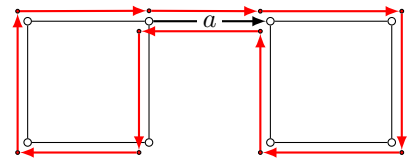
\begin{figure}[htbp]
\begin{center}
\subfloat[Joining two 4-cycles at a vertex\label{fig:4cyclesAtVtx} results in a walk of length 8]{
\begin{tikzpicture}[vtx/.style={draw, circle, inner sep = 1 pt, font = \tiny},lbl/.style = {midway, fill = white, inner sep = 2 pt, font=\footnotesize}, node distance = 1.5cm  ]
\foreach \j in {0,1,2,3}{
\node[vtx] (v\j) at (360*\j/4:1){};
}
\foreach \j in {0,1,3}{
\node[vtx] (w\j) at ($(360*\j/4:1)+(2,0)$){};
}
\foreach \j in {v,w}{
\draw[-latex] (\j0) -- node[lbl]{$a$} (\j1);
\draw[] (\j3) --  (\j0);
}
\draw[-latex] (v1) -- node[lbl]{$b$} (v2);
\draw[-latex] (w1) -- node[lbl, pos = .2]{$b$} (v0);
\draw[-latex] (v2) -- node[lbl,]{$c$} (v3);
\draw[-latex] (v0) -- node[lbl, pos = .4]{$c$} (w3);

\begin{scope}
\tikzset{node distance = .5mm, vtxx/.style={vtx, inner sep = .5pt, fill = red}}
\node[vtxx, above right=2mm and .5 mm of v0] (a)  {};
\node[vtxx,fill = red, above =1 mm of v1] (b){};
\node[vtxx,  left = 1mm of v2](c){};
\node[vtxx, below  = 1mm of v3] (d){};
\node[vtxx, below  = 1.5mm of v0] (e){};
\node[vtxx, below  = 1mm of w3] (f){};
\node[vtxx, right  = 1mm of w0] (g){};
\node[vtxx, above = 1mm of w1] (h){};
\foreach \j/\k in {a/b, b/c, c/d, d/e,h/g, g/f}{
\draw[-latex,red, thick] (\j) -- (\k);
}
\draw[-latex, red, thick] (e) to[out = 45, in = -90-45] (h);
\draw[-latex, red, thick] (f) to[out = 180-45, in = -45, ] (a);
\end{scope}

\end{tikzpicture}
}
\hspace{.5cm}
\subfloat[joining two 4-cycles along an edge/arc results in a walk of length 8\label{fig:4cyclesAtEdge}]{
\begin{tikzpicture}[vtx/.style={draw, circle, inner sep = 1 pt, font = \tiny},lbl/.style = {midway, fill = white, inner sep = 2 pt, font=\footnotesize}, node distance = 1.5cm  ]
\node[vtx] (v1) at (0,0){};
\node[vtx,below = of v1] (v2) {}; 
\node[vtx, right  = of v2] (v3){};
\node[vtx, above = of v3] (v4){};
\node[vtx, right = of v4] (v5){};
\node[vtx, below = of v5] (v6){};
\draw (v1) -- (v2) -- (v3) -- (v6) -- (v5) -- (v4) --(v1);
\draw[thick, -latex] (v4) -- node[lbl]{$a$} (v3);

\begin{scope}
\tikzset{node distance = .5mm, vtxx/.style={vtx, inner sep = .5pt, fill = red}}
\node[vtxx, above=of v4] (a)  {};
\node[vtxx,fill = red, above left=1 mm of v1] (b){};
\node[vtxx, below left = 1mm of v2](c){};
\node[vtxx, below left = 1mm of v3] (d){};
\node[vtxx, below left = 1mm of v4] (e){};
\node[vtxx, above right = 1mm of v5] (f){};
\node[vtxx, below right = 1mm of v6] (g){};
\node[vtxx, below right = 1mm of v3] (h){};
\foreach \j/\k in {e/d,d/c,c/b,b/a,  h/g, g/f, f/e}{
\draw[-latex,red, thick] (\j) -- (\k);
}
\draw[-latex, red, thick] (a) to[out = 0, in = 90, looseness = .5] (h);
\end{scope}

\end{tikzpicture}
}
\hspace{.5cm}
\subfloat[Joining two 4-cycles with a path of length 1 results in a walk of length 10\label{fig:4cyclesJoinEdge}]{

\begin{tikzpicture}[vtx/.style={draw, circle, inner sep = 1 pt, font = \tiny},lbl/.style = {midway, fill = white, inner sep = 2 pt, font=\footnotesize}, node distance = 1.5cm  ]
\node[vtx] (v1) at (0,0){};
\node[vtx,below = of v1] (v2) {}; 
\node[vtx, right  = of v2] (v3){};
\node[vtx, above = of v3] (v4){};
\node[vtx, right = of v4] (v44){};
\node[vtx, right = of v3] (v33){};
\node[vtx, right = of v44] (v5){};
\node[vtx, below = of v5] (v6){};
\draw (v4) -- (v3) -- (v2) -- (v1) -- (v4);
\draw (v44) -- (v5) -- (v6) -- (v33) -- (v44);
\draw[thick, -latex] (v4) -- node[lbl]{$a$} (v44);

\begin{scope}
\tikzset{node distance = .5mm, vtxx/.style={vtx, inner sep = .5pt, fill = red}}
\node[vtxx, above =of v4] (a)  {};
\node[vtxx,fill = red, above left=1 mm of v1] (b){};
\node[vtxx, below left = 1mm of v2](c){};
\node[vtxx, below left = 1mm of v3] (d){};
\node[vtxx, below left = 1mm of v4] (e){};
\node[vtxx, above right = 1mm of v5] (f){};
\node[vtxx, below right = 1mm of v6] (g){};
\node[vtxx, below left = 1mm of v33] (h){};
\node[vtxx, below left = 1mm of v44] (q){};
\node[vtxx, above left = 1mm of v44] (s){};
\foreach \j/\k in {e/d,d/c,c/b,b/a, a/s,s/f, f/g, g/h, h/q}{
\draw[-latex,red, thick] (\j) -- (\k);
}
\draw[-latex, red, thick, ] (q) to[] (e);
\end{scope}

\end{tikzpicture}
}
\caption{Examples of short closed walks formed by traversing joined cycles in various ways.
}
\label{fig:joinCycles}
\end{center}
\end{figure}


There are limited ways to construct non-cycle closed walks of lengths 6, 8 or 10. It suffices to look at the following types of joined cycles; if these types of joined cycles exist in the graph, then the voltage sums are analyzed separately along the walks.
\begin{remark}\label{lem:nonCycleWalks} Non-cycle, non-reversing, non-lollipop closed walks of length 6, 8, or 10 in bipartite graphs have as their underlying structure the following forms:
\bi
\item A 4-cycle and a 2-cycle joined at a vertex may lift to a 6-cycle;
\item A 4-cycle and a 2-cycle joined at a shared edge may lift to a 6-cycle;
\item pairs of cycles of length 4 joined at a vertex may lift to an 8-cycle; 
\item pairs of cycles of length 4 joined at a shared edge may lift to an 8-cycle; 
\item a 4-cycle and a 6-cycle joined at a vertex may lift to a 10-cycle;
\item two 4-cycles joined by a path of length 1 may lift to a 10-cycle; 
\item a 4-cycle and a 6-cycle joined at a shared edge may lift to a 10-cycle; 
\ei
\end{remark}

Other walks formed by going along cycles joined by shared edges or along cycles joined by paths have lengths larger than 10. Note that the voltage graphs we consider later in this work that potentially have girth more than 6 do not have 2-cycles.




 In the rest of the paper we will construct families of semicubic graphs of girth $g\in \{6,8,10,12\}$. To obtain these graphs, we begin by  describing two different families of voltage graphs called $\mathcal{G}_{4t}$, $t\geq 2$, and $\mathcal{G}_{4t+2}$ for $t\geq 1$. We use these graphs for $t\in \{2,3\}$ in the first case and for $t\in\{1,2\}$ in the second case to obtain families of semicubic graphs with girth $g=4t$ and $g=4t+2$. \\

\section{A family of semicubic graph of girth $4t+2$}\label{case4t+2}

In this section we we will use a family of voltage graphs $\mathcal{G}_{4t+2}$ to construct a family of $(\{3;m \};4t+2)$-biregular graphs called  $(\mathcal{G}_{4t+2}; m)$. First of all, we describe the construction of one of the graphs of this family, which we will call $G_{4t+2}$. This voltage graph begins with two trees $X_{t}$ and $Y_{t}$, which are identical except for their name. 
Graph $X_{t}$
begins with a vertex $x^{*}$. We join this vertex to a single vertex $x$ and construct a binary tree from $x$. Vertex $x$ has two children, $x_{0}$ and $x_{1}$; vertex $x_{i}$, $i = 0,1$ has two children $x_{i0}$ and $x_{i1}$, and in general, for each given bit string $b$ of length $\ell$, $1\leq \ell \leq 2t-1$, vertex $x_{b}$ has children $x_{b0}$ and $x_{b1}$. Including $x^{*}$ and $x$, the total height of this tree is $2t$.



Next, we delete all of the left-handed children of vertex $x_{a}$ where $a = \underbrace{0\cdots0}_{t-1}$, that is, all the vertices whose bit strings are of the form $\underbrace{0\cdots0}_{t-1}\ell$ for nonempty bitstring $\ell$. For example, when $t = 2$, we delete the children of $x_{0}$, and $x_{0}$ has a bitstring of length 1. The remaining leaves of the tree have bit strings of length $2t-1$. Observe that by construction, the distance from  $x^{*}$ to $x_{a}$ is $t$ and the distance from $x_{a}$ to the level of the leaves is also $t$.
The vertex $x^{*}$ and the $2^{2t-1}-2^{t-1}=2^{t-1}(2^{t}-1)$ leaves all have degree 1; the remaining internal vertices have degree 3.
This pruned tree is called $X_{t}$. To continue,  we  take a copy of $X_{t}$, flip it vertically, change the labels from $x$ to $y$, and call it $Y_{t}$. Finally, we join the trees $X_{t}$ and $Y_{t}$ using a zero-labeled edge between $x_a$ and $y_a$. This tree is called $T_{4t+2}$, where the "$4t+2$" indicates the height of the tree. See Figure \ref{fig:VtreeAndGtree}.

Next, we construct a family of voltage graphs $\mathcal{G}_{4t+2}$ for $t=1, 2$  by adding edge-labeled arcs to $T_{4t+2}$ 
 between the leaves of $X$ and $Y$ in such a way that each leaf (excluding $x^{*}$ and $y^{*}$) is incident with one edge of $T_{4t+2}$ and two introduced arcs; that is so that all the vertices, except the pinned vertices $x^{*}$ and $y^{*}$, have degree 3. A particular element in this family (a particular assignment of arcs and voltages) will be denoted $G_{4t+2}$. Note that this family contains voltage graphs with different arc assignments, and, for each arc assignment, different labels. For a specific choice of $G_{4t+2}$ (that is, a particular choice of arcs and labels),
see Figure \ref{fig:G10-voltage}.

\def\xx{24}
\def\yy{13}
\begin{figure}[htbp]
\begin{center}\label{fig:treeX}
\subfloat[The tree $X_{2}$. The light gray subgraph has been pruned.]{
\begin{tikzpicture}[vtx/.style={draw, circle, inner sep = 1 pt, font = \tiny},lbl/.style = {midway, fill = white, inner sep = 2 pt, font=\footnotesize}]
\node[draw, inner sep = 1 pt, font = \tiny] (xx) {$x^{*}$};
\node[vtx, below = \yy pt of xx] (x){};
\node[vtx, below left = \yy pt and \xx pt of x] (x0) {};
\node[vtx, below right = \yy pt and \xx pt of x] (x1){}; 
\node[vtx, below left = \yy pt and \xx/2 pt  of x0, opacity = .2] (x00) {};
\node[vtx, below right = \yy pt and \xx/2 pt  of x0] (x01) {};
\node[vtx, below left = \yy pt and \xx/2 pt  of x1] (x10) {};
\node[vtx, below right = \yy pt and \xx/2 pt  of x1] (x11) {};

\node[vtx, below left = \yy pt and \xx/4 pt  of x00, opacity = .2] (x000) {};
\node[vtx, below left = \yy pt and \xx/4 pt  of x01] (x010) {};
\node[vtx, below left = \yy pt and \xx/4 pt  of x10] (x100) {};
\node[vtx, below left = \yy pt and \xx/4 pt  of x11] (x110) {};

\node[vtx, below right = \yy pt and \xx/4 pt  of x00, opacity = .2] (x001) {};
\node[vtx, below right = \yy pt and \xx/4 pt  of x01] (x011) {};
\node[vtx, below right = \yy pt and \xx/4 pt  of x10] (x101) {};
\node[vtx, below right = \yy pt and \xx/4 pt  of x11] (x111) {};

\draw (xx) -- (x);
\foreach \i in {0,1}{
\draw (x) -- (x\i);}
\foreach \j in {0,1}{
\draw (x1) -- (x1\j);
\draw (x10) -- (x10\j);
\draw (x01) -- (x01\j);
\draw (x11) -- (x11\j);
}
\draw (x0) -- (x01);
\draw[opacity = .2] (x0)-- (x00);
\draw[opacity = .2] (x00)-- (x000);
\draw[opacity = .2] (x00)-- (x001);


\end{tikzpicture}
}
\hspace{2cm}
\subfloat[The tree $T_{4t+2}$ for $t = 2$]{
\begin{tikzpicture}[vtx/.style={draw, circle, inner sep = 1 pt, font = \tiny},lbl/.style = {midway, fill = white, inner sep = 2 pt, font=\footnotesize}]
\node[draw, inner sep = 1 pt, font = \tiny] (xx) {$x^{*}$};
\node[vtx, below = \yy pt of xx] (x){};
\node[vtx, below left = \yy pt and \xx pt of x] (x0) {};
\node[vtx, below right = \yy pt and \xx pt of x] (x1){}; 
\node[vtx, below right = \yy pt and \xx/2 pt  of x0] (x01) {};
\node[vtx, below left = \yy pt and \xx/2 pt  of x1] (x10) {};
\node[vtx, below right = \yy pt and \xx/2 pt  of x1] (x11) {};

\node[vtx, below left = \yy pt and \xx/4 pt  of x01] (x010) {};
\node[vtx, below left = \yy pt and \xx/4 pt  of x10] (x100) {};
\node[vtx, below left = \yy pt and \xx/4 pt  of x11] (x110) {};

\node[vtx, below right = \yy pt and \xx/4 pt  of x01] (x011) {};
\node[vtx, below right = \yy pt and \xx/4 pt  of x10] (x101) {};
\node[vtx, below right = \yy pt and \xx/4 pt  of x11] (x111) {};


\node[draw, inner sep = 1 pt, font = \tiny, below = 11*\yy pt of xx] (yy) {$y^{*}$};
\node[vtx, above = \yy pt of yy] (y){};
\node[vtx, above left = \yy pt and \xx pt of y] (y0) {};
\node[vtx, above right = \yy pt and \xx pt of y] (y1){}; 
\node[vtx, above right = \yy pt and \xx/2 pt  of y0] (y01) {};
\node[vtx, above left = \yy pt and \xx/2 pt  of y1] (y10) {};
\node[vtx, above right = \yy pt and \xx/2 pt  of y1] (y11) {};

\node[vtx, above left = \yy pt and \xx/4 pt  of y01] (y010) {};
\node[vtx, above left = \yy pt and \xx/4 pt  of y10] (y100) {};
\node[vtx, above left = \yy pt and \xx/4 pt  of y11] (y110) {};

\node[vtx, above right = \yy pt and \xx/4 pt  of y01] (y011) {};
\node[vtx, above right = \yy pt and \xx/4 pt  of y10] (y101) {};
\node[vtx, above right = \yy pt and \xx/4 pt  of y11] (y111) {};

\draw (xx) -- (x);
\foreach \i in {0,1}{
\draw (x) -- (x\i);}
\foreach \j in {0,1}{
\draw (x1) -- (x1\j);
\draw (x10) -- (x10\j);
\draw (x01) -- (x01\j);
\draw (x11) -- (x11\j);
}
\draw (x0) -- (x01);

\draw (yy) -- (y);
\foreach \i in {0,1}{
\draw (y) -- (y\i);}
\foreach \j in {0,1}{
\draw (y1) -- (y1\j);
\draw (y10) -- (y10\j);
\draw (y01) -- (y01\j);
\draw (y11) -- (y11\j);
}
\draw (y0) -- (y01);

\draw[->, thick] (x0) to[bend right = 45] (y0);



\end{tikzpicture}
}
\caption{The trees $X_{2}$ and $T_{4t+2}$, for $t = 2$}
\label{fig:VtreeAndGtree}
\end{center}
\end{figure}
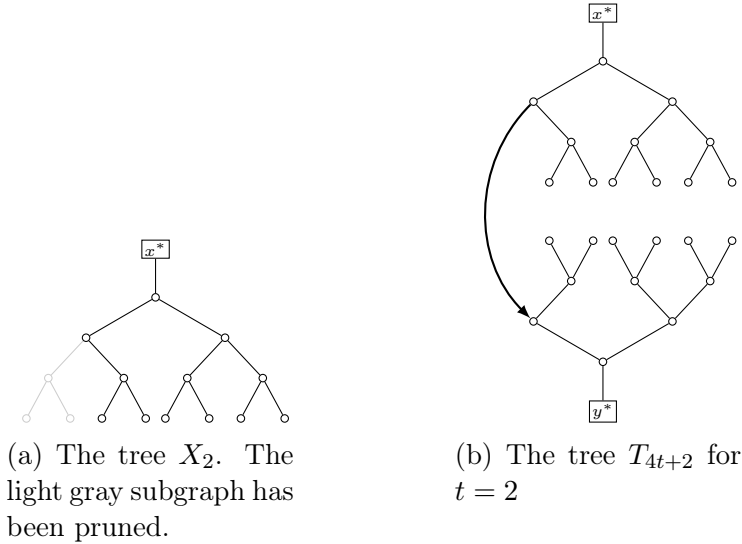

We let $(G_{4t+2}, m)$ denote the $\mathbb{Z}_{m}$ lift of a graph $G_{4t+2}$. 
The roots $x^{*}$ and $y^{*}$ in $G_{4t+2}$ are pinned vertices $x^*$ and $y^*$ in $(G_{4t+2}; m)$, and they have degree $m$. 

The graph $(G_{4t+2}; m)$ has $2+2m(\sum_{i=0}^{2t-1}2^i - \sum_{i=0}^{t-1}2^i) =2+2m(\sum_{i=t}^{2t-1}2^i)$ vertices, found by counting the two pinned vertices $x^{*}$ and $y^{*}$, summing the vertices in the pruned binary trees $X_{t}$ and $Y_{t}$, and then multiplying those vertices by $m$ from the lift. 

By choosing certain arcs and certain voltage assignments in $(\mathcal{G}_{4t+2}, m)$, we obtain specific families of $(\{3;m \};4t+2)$-biregular graphs for $t=1$ and $t=2$. 


Recall that $T_{4t+2}$ is a voltage tree 
obtained by joining the binary trees $X_{t}$ and $Y_{t}$ by the edge $(x_a y_a)$, and all the edges, including $(x_ay_a)$, have voltage assignments equal to $0$. Let $(T_{4t+2}; m)$ be the derived graph of $T_{4t+2}$, with $x^{*}$ and $y^{*}$ as pinned vertices. Notice that this derived graph has the same order as $(G_{4t+2}; m)$, that is, $n(T_{4t+2}; m) = n(G_{4t+2}; m) = 2+2m(\sum_{j=t}^{2t-1}2^j)$. 

Applying Lemma \ref{lemma:pinnedwalks} we can conclude the following:


\begin{lemma}\label{txygirth}
$(T_{4t+2}; m)$ has girth $4t+2$.
\end{lemma}

\begin{proof}
By construction, since the distance in $T_{4t+2}$ between $x^{*}$ and $x_{a}$ is $t$, there exists a path $P=(x^{*}, \dotsc, x_a, y_a, \dotsc,  y^{*})$ of length $2t+1$ between the two roots in the voltage graph $T_{4t+2}$. This path lifts to a cycle of length $4t+2$ in $(T_{4t+2}; m)$. No other shorter paths in $T_{4t+2}$ lift to cycles, since $T_{4t+2}$ is a tree. 
Therefore, the girth of $(T_{4t+2}; m)$ is equal to $4t+2$.
\end{proof}

Figure \ref{fig:cycleLifts} shows an example of such a cycle when $t = 1$.

\begin{lemma}\label{girth4t+2}
The graph $(G_{4t+2}; m)$ has girth at most $4t+2$.  
\end{lemma}
\begin{proof}
The graph $(T_{4t+2},m)$ is a subgraph of $(G_{4t+2},m)$.
\end{proof}

\begin{lemma}\label{bipartite}
The graph $\mc{G}_{4t+2}$ is bipartite.
\end{lemma}
\begin{proof}Let $\mathcal{X}_j$ be the set of vertices of the pruned binary tree $X_t$ contained in the voltage graph $G_{4t+2}$ at distance $j$ from $x^{*}$: $\mathcal{X}_{0}=x^{*}$, $\mathcal{X}_{1}=x$, $\mathcal{X}_{2}=\{x_0, x_1\}$, $\mathcal{X}_{3}=\{x_{00}, x_{01}, x_{10}, x_{11}\}$ and so on. Analogously, let $\mathcal{Y}_j$ be the set of vertices of the pruned binary tree $Y_t$ at distance $j$ of $y^{*}$. Observe that  $B_{\mathcal{X}^{'}}=\{\mathcal{X}_0, \mathcal{X}_2, ..., \mathcal{X}_{2t}\}$ and $B_{\mathcal{X}^{''}}=\{\mathcal{X}_1, \mathcal{X}_3, ...,\mathcal{X}_{2t-1}\}$ are two disjoint subsets of vertices of $X_t$ whose union is $V(X_t)$. Analogously, $B_{\mathcal{Y}^{'}}$ and $B_{\mathcal{Y}^{''}}$ are two disjoint subsets of the vertices of $Y_t$ which union is $V(Y_t)$. The bipartite classes are $\mathcal{B}_1=B_{\mathcal{X}^{'}}\cup B_{\mathcal{Y}^{''}}$ and $\mathcal{B}_2=B_{\mathcal{X}^{''}}\cup B_{\mathcal{Y}^{'}}$. 
\end{proof}

Consequently, by Lemma \ref{bipartite}, we conclude that all derived graphs $(G_{4t+2}; m)$ in the family of graphs $(\mathcal{G}_{4t+2}, m)$  are also bipartite graphs. 


\section{A family of semicubic cages of girth $6$}
In this section, we obtain a family of semicubic cages of girth $6$, by introducing a family of voltage graphs $G_6$, shown in Figure \ref{fig:G6}, for which the $(G_{6},m)$-graphs attain the lower bound given in \cite{HWJ92}. If the voltage assignments of the graphs as described in Theorem \ref{G6} are $\alpha=1$ and $\beta=m-1$ we obtain the graphs given in \cite{HWJ92}. 

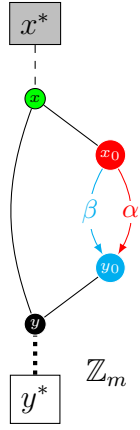
\begin{figure}[htbp]
\begin{center}
\begin{tikzpicture}[vtx/.style={draw, circle, inner sep = 1 pt, font = \tiny},lbl/.style = {midway, fill = white, inner sep = 2 pt, font=\footnotesize}]
\node[draw] (yy) at (0,0) {$y^{*}$};
\node[vtx, white,fill = black] (y) at (0,1) {$y$};
\node[vtx, white,fill = cyan] (y0) at (1,1.75) {$y_{0}$};
\node[vtx, white, fill = red] (x0) at (1,3.25) {$x_{0}$};
\node[vtx, fill = green] (x) at (0,4) {$x$};
\node[draw, fill = gray!50] (xx) at (0,5) {$x^{*}$};
\draw[] (x) to[bend right = 20] (y);
\draw[dashed] (xx) to (x);
\draw[dotted, ultra thick] (yy) to (y);
\draw[] (y) to (y0);
\draw[] (x) to (x0);
\draw[-latex, red] (x0) to[bend left = 30] node[lbl]{$\alpha$} (y0);
\draw[-latex, cyan] (x0) to[bend right = 30] node[lbl]{$\beta$} (y0);
\path node[below right = 6 pt and 12 pt of y] {$\mathbb{Z}_{m}$};
\end{tikzpicture}
\end{center}

%
   \caption{Graph $G_{6}$ is girth 6 for $m \geq 3$ and any $\alpha, \beta$ such that $\alpha\neq \beta$, $2(\alpha - \beta) \not\equiv 0 \bmod m$, $\alpha\neq 0$ and $\beta \neq 0$.}
    \label{fig:G6}
\end{figure}

 \begin{theorem}\label{G6}
The family of  graphs $(G_6,m)$ given by the family of voltage graphs $G_6$ with voltage assignments $(x_0 y_0)=\alpha $, $(x_0 y_0)=\beta$, $\alpha\neq \beta$, $\alpha\neq 0$ and $\beta \neq 0$, and $2(\alpha - \beta) \not\equiv 0 \mod m$  for $m\geq 3$ has girth $6$. Moreover, it is a $(\{3,m\};6)$-graph with two vertices of degree $m$ and $4m$ vertices of degree $3$. 

\end{theorem}

%
%

 \begin{proof}

We analyze the following subgraphs of the voltage graph: paths that contain two pinned vertices,  ``lollipop'' walks containing a single cycle joined to a path, cycles in $G_6$ without pinned vertices, and non-reversing non-cycle closed walks whose total length is 6. 

\begin{enumerate}
\item A minimum path between $x^{*}$ and $y^{*}$ in the voltage graph $G_{6}$ is the path $(x^{*} x y y^{*})$, which lifts to the 6-cycle $(x^{*} x^{0} y^{0} y^{*} y^{1} x^{1} x^{*})$, so the girth is at least 6.

\item All the lollipops in $G_6$ with some edges with non-zero voltage assignments are essentially of two types: one has a path of length 1 and one $4$-cycle, like $(x^{*},x,y,y_0,x_0,x,x^{*})$, and one has a path of length 2 with a $2$-cycle, like $(x^{*},x,x_0,y_0,x_0,x,x^{*})$. By Lemma \ref{lollipop} both induce $6$-cycles in $(G_6; m)$.

\item In considering the behavior of lifts of cycles without pinned vertices, note that since the graphs in $(\mathcal{G}_6; m)$ are bipartite, it is sufficient to show that every cycle of length $2$ or $4$ in $G_6$ lifts to a cycle of length at least $6$ in $(G_6; m)$. 



Inspection shows that there are 2 different 4-cycles, one of the form $(x, \underbrace{x_{0}, y_{0}}_{\alpha}, y, x)$ and one of the form $(x, \underbrace{x_{0}, y_{0}}_{\beta}, y, x)$ where the bracing indicates the voltage label on the directed edge $(x_{0}, y_{0})$ (and which arrow we are choosing). The walk formed by traversing this cycle once lifts to a 4-cycle only when $\alpha \equiv 0 \bmod m$ or $\beta \equiv 0 \bmod m$. There is a single 2-cycle, using both directed arrows (reversing one), corresponding to the directed cycle 
$\overunderbraces{&\br{2}{\alpha} &  }%
  {&x_{0}, &y_{0}, &x_{0}&}%
  {& &\br{3}{-\beta}}$
 %
(note that we negate the voltage since we are traveling backwards along the arrow). Once around this 2-cycle lifts to a 2-cycle when $\alpha =\beta$, and twice around the 2-cycle lifts to a 4-cycle when $2(\alpha - \beta) \equiv 0 \bmod m$, since in these cases the voltage sum for the closed walk equals 0.


\item The final option for walks that could lift to cycles consists of cycles joined by vertices or paths or that share an edge or path. However, since the only available cycles in $G_{6}$ are 2-cycles or 4-cycles, such walks are all of length greater than or equal to 6, so the lifts of these walks do not decrease the girth.
\end{enumerate}

\end{proof}

The order of the graph for $m\geq 3$ is equal to $4m+2$. This family is isomorphic to the family constructed in \cite{ABLM13}, and it also appears in \cite{AEJ16} in the general Theorem for $g\equiv 2 \pmod 4$. 
\section{A family of semicubic graphs of girth $10$}\label{girth10}

In this section, we will give a voltage assignment for a particular voltage graph $G_{10}$ in  $\mc{G}_{4t+2}$ for $t=2$, shown in Figure \ref{fig:G10-voltage}, and we will prove that its $\mathbb{Z}_{m}$ lifts form a family of semicubic graphs of girth $10$. The arcs between leaves of $\mc{G}_{10}$ and voltage assignments on those arcs that produce the voltage graph $G_{10}$ are also given in Table \ref{table:G10table}. 

\def\xx{70}
\def\yy{20}
\begin{figure}[htbp]
\begin{center}
\begin{tikzpicture}[vtx/.style={draw, circle, inner sep = 1 pt, font = \tiny},bigvtx/.style={draw, , inner sep = 1 pt, font = \tiny, minimum height = 10 pt},lbl/.style = {midway, fill = white, inner sep = 1 pt, font=\scriptsize}]
\node[draw, inner sep = 1 pt, font = \tiny] (xx) {$x^{*}$};
\node[vtx, below = \yy pt of xx] (x){};
\node[vtx, below left = \yy pt and \xx pt of x] (x0) {};
\node[vtx, below right = \yy pt and \xx pt of x] (x1){}; 
\node[vtx, below right = \yy pt and \xx/2 pt  of x0] (x01) {};
\node[vtx, below left = \yy pt and \xx/2 pt  of x1] (x10) {};
\node[vtx, below right = \yy pt and \xx/2 pt  of x1] (x11) {};

\node[vtx, below left = \yy pt and \xx/4 pt  of x01, label={[font = \tiny, label distance=-5pt]above left:$x_{010}$}
] (x010) {};
\node[vtx, below left = \yy pt and \xx/4 pt  of x10, label={[font = \tiny, label distance=-5pt]above left:$x_{100}$}] (x100) {};
\node[vtx, below left = \yy pt and \xx/4 pt  of x11, label={[font = \tiny, label distance=-5pt]above left:$x_{110}$}] (x110) {};

\node[vtx, below right = \yy pt and \xx/4 pt  of x01,label={[font = \tiny, label distance=-5pt]above left:$x_{011}$}] (x011) {};
\node[vtx, below right = \yy pt and \xx/4 pt  of x10, label={[font = \tiny, label distance=-5pt]above left:$x_{101}$}] (x101) {};
\node[vtx, below right = \yy pt and \xx/4 pt  of x11, label={[font = \tiny, label distance=-5pt]above right:$x_{111}$}] (x111) {};


\node[draw, inner sep = 1 pt, font = \tiny, below = 12*\yy pt of xx] (yy) {$y^{*}$};
\node[vtx, above = \yy pt of yy] (y){};
\node[vtx, above left = \yy pt and \xx pt of y] (y0) {};
\node[vtx, above right = \yy pt and \xx pt of y] (y1){}; 
\node[vtx, above right = \yy pt and \xx/2 pt  of y0] (y01) {};
\node[vtx, above left = \yy pt and \xx/2 pt  of y1] (y10) {};
\node[vtx, above right = \yy pt and \xx/2 pt  of y1] (y11) {};

\node[vtx, above left = \yy pt and \xx/4 pt  of y01, label={[font = \tiny, label distance=-5pt]below left:$y_{010}$}] (y010) {};
\node[vtx, above left = \yy pt and \xx/4 pt  of y10, label={[font = \tiny, label distance=-5pt]below left:$y_{100}$}] (y100) {};
\node[vtx, above left = \yy pt and \xx/4 pt  of y11, label={[font = \tiny, label distance=-5pt]below left:$y_{110}$}] (y110) {};

\node[vtx, above right = \yy pt and \xx/4 pt  of y01,label={[font = \tiny, label distance=-5pt]below left:$y_{011}$}] (y011) {};
\node[vtx, above right = \yy pt and \xx/4 pt  of y10,label={[font = \tiny, label distance=-5pt]below left:$y_{101}$}] (y101) {};
\node[vtx, above right = \yy pt and \xx/4 pt  of y11, label={[font = \tiny, label distance=-5pt]below right:$y_{111}$}] (y111) {};

\draw (xx) -- (x);
\foreach \i in {0,1}{
\draw (x) -- (x\i);}
\foreach \j in {0,1}{
\draw (x1) -- (x1\j);
\draw (x10) -- (x10\j);
\draw (x01) -- (x01\j);
\draw (x11) -- (x11\j);
}
\draw (x0) -- (x01);

\draw (yy) -- (y);
\foreach \i in {0,1}{
\draw (y) -- (y\i);}
\foreach \j in {0,1}{
\draw (y1) -- (y1\j);
\draw (y10) -- (y10\j);
\draw (y01) -- (y01\j);
\draw (y11) -- (y11\j);
}
\draw (y0) -- (y01);

\draw[-latex] (x0) to[bend right = 45] (y0);

\draw[-latex, red,  thick] (x010) to[bend right = 0] node[lbl]{1} (y010);
\draw[-latex, blue,  thick] (x010) to[bend left = 0] node[lbl, pos = .1]{2} (y111);
\draw[-latex, blue,  thick] (x011) to[bend right = 0] node[lbl, pos = .6]{2} (y011);
\draw[-latex, red,  thick] (x011) to[bend right = 0] node[lbl, near start,]{1} (y110);
\draw[-latex, blue,  thick] (x100) to[bend right = 0] node[lbl, near end,]{2} (y010);
\draw[-latex, red,  thick] (x100) to[bend right = 0] node[lbl, pos = .65,]{1} (y101);
\draw[-latex, red,  thick] (x101) to[bend right = 0] node[lbl, near end,]{1} (y011);
\draw[-latex, blue,  thick] (x101) to[bend right = 0] node[lbl, near end,]{2} (y100);
\draw[-latex, blue,  thick] (x110) to[bend right = 0] node[lbl, near end,]{2} (y110);
\draw[-latex, red,  thick] (x110) to[bend right = 0] node[lbl, near end,]{1} (y111);
\draw[-latex, green!60!black,  thick] (x111) to[bend right = 0] node[lbl, pos = .45]{3} (y100);
\draw[-latex, ,  thick] (x111) to[bend right = 0] node[,lbl,pos = .6 , inner sep = .5]{0} (y101);

\begin{scope}[on background layer]
\draw[line width = 5 pt, opacity = .5, yellow] (x111.center) to[bend right = 0] (y100.center) -- (y10) -- (y101)  to[bend left = 0] (x111.center);
\draw[line width = 5 pt, opacity = .5, yellow] (x110.center) to[bend right = 0] (y110.center) -- (y11) -- (y111)  to[bend left = 0] (x110.center);
\end{scope}

\end{tikzpicture}
\caption{The voltage graph $G_{10}$. Unlabeled edges have voltage 0. Edges with the same voltage assignment are colored the same color for convenience. The two 4-cycles in the graph are highlighted in yellow.}
\label{fig:G10-voltage}
\end{center}
\end{figure}
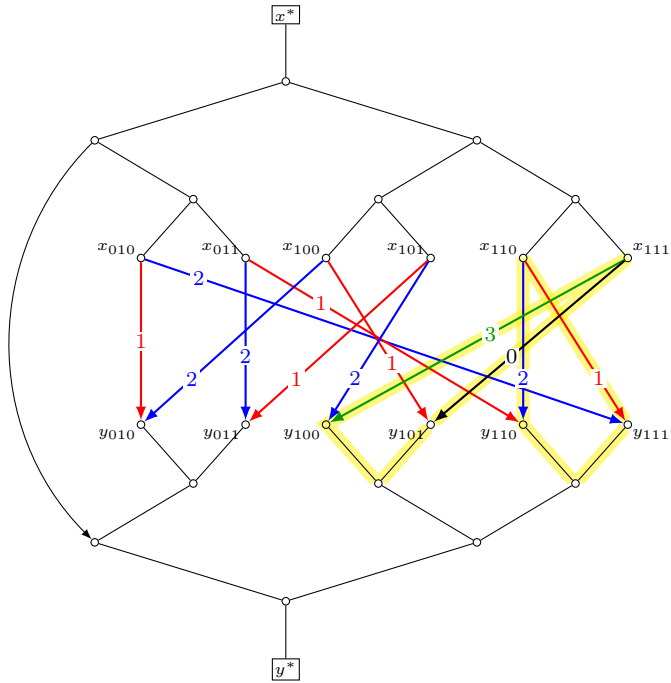



\begin{theorem}\label{G10old}
The derived graphs $(G_{10}, m)$ associated with the voltage graph $G_{10}$ (shown in Figure \ref{fig:G10-voltage}) have girth 10 for all $m\geq 4$ except for $m = 6$. Moreover, the graphs $(G_{10}, m)$ are $(\{3,m\};10)$-graphs with $2$ vertices of degree $m$ and $24m$ vertices of degree $3$.
\end{theorem}

\begin{table}[htp]
\caption{Labelled arcs between leaves of $\mc{G}_{10}$ forming a voltage graph $G_{10}$. The $\mathbb{Z}_{m}$ lifts $(G_{10}, m)$ are girth 10 for $m \neq 1,2,3,6$.}
\begin{center}
\begin{tabular}{c|  c  c c c c c c c c c c c}
Starting leaf & $x_{010}$ & $x_{010}$ & $x_{011}$ & $x_{011}$ & $x_{100}$ & $x_{100}$ & $x_{101}$ & $x_{101}$  & $x_{110}$ & $x_{110}$ & $x_{111}$ & $x_{111}$ \\
Ending leaf & $y_{010}$ & $y_{111}$ & $y_{011}$ & $ y_{110}$ & $y_{010}$& $y_{101}$ & $y_{011}$& $y_{100}$ & $y_{110}$& $y_{111}$ & $y_{100}$& $y_{101}$\\
Voltage & 1 & 2 & 2 & 1 & 2 & 1 & 1 & 2 & 2 & 1 & 3 & 0 \\
\end{tabular}
\end{center}
\label{table:G10table}
\end{table}%

\begin{proof}
It suffices to consider the sums of voltages on non-reversing walks in the voltage graph $G_{10}$; a walk lifts to a cycle in $(G_{10},m)$ if and only if the sum of the voltages along the walk is congruent to $0 \pmod m$. 

Observe that by 
Lemma \ref{girth4t+2},
the girth of $(G_{10}, m)$ is at most 10. 
To show that the girth is equal to 10, we need to show that there are no cycles of length 2, 4, 6, or 8 (there are no odd cycles because $G_{10}$ is bipartite, by Lemma \ref{bipartite}).

We consider the following walks to determine if they lift to cycles: (1) cycles; (2) walks formed by combining cycles; (3) lollipops.
\be
\item There are no non-reversing walks of length 2 in $G_{10}$; the only such cycle would be going along an edge and immediately returning back along the same edge, which is a reversing walk.
\item Analysis in Mathematica \cite{Mathematica} of the voltage graph $G_{10}$ (with all edges doubled, to allow for travel in either direction along the edges) gives 82 cycles of length 4, 6, or 8. Summing the voltages along all these cycles results in voltage sums in the set $\{-1,-2,-3,1,2,3\}$. 
Since $m >3$, one time around each of these individual cycles lifts to a path that is not a closed cycle, since the voltage sums are not congruent to $0 \mod m$. However, an inspection of the two $4$-cycles shows that twice around the 4-cycle $(\overbrace{x_{111},y_{110}}^3, y_{11}, y_{101}, x_{111})$, which has voltage sum 3, lifts to an 8-cycle when $m = 6$, since $2\cdot 3 \equiv 0 \pmod 6$. Thus, we conclude that $(G_{10}, 6)$ has girth at most 8, so $m = 6$ is excluded.

\item Now we consider walks that are not themselves cycles that are formed by joining cycles at vertices, edges, or paths, or joining two cycles by a path There are no such walks of length 4 (all walks of length 4 in $G_{10}$ are cycles). Following the discussion in Observation \ref{lem:nonCycleWalks}, note that walks of length 6 must be formed by joining 2-cycles and 4-cycles, but there are no 2-cycles. Closed walks that are of length 8 whose voltages sum to 0 could be formed by joining two 4-cycles at a vertex, or by joining a 4-cycle and a 4-cycle along a (possibly non-zero-labeled) edge 
(See Figure \ref{fig:joinCycles}). However, there are only two 4-cycles in $G_{10}$ and they are disjoint (see Figure \ref{fig:G10-voltage}, where the two 4-cycles are highlighted in yellow), so there are no such walks. Joining disjoint cycles with a path, as in Figure \ref{fig:4cyclesJoinEdge} results in even longer cycles in the lift.

\item Finally, we consider lollipop walks formed by joining a path to a pinned vertex. None of these walks lift to short cycles, because the pinned vertices are too far from the cycles. For example, a shortest path to a cycle is $(x^{*}, x,x_{0})$ joined to the 6-cycle $(x_{0}, y_{0}, y_{01}, \overbrace{y_{010},x_{010}}^{-1}, x_{01}, x_{0})$, which will lift to a cycle of length 10. Similarly, a path of length 4 joined to a 4-cycle will lift to a cycle of length 12.
\ee
In all cases, as long as $m \neq 1,2,3, 6$, there are no short walks in $G_{10}$ that lift to short cycles (with length at most 8) in the derived graph $(G_{10}, m)$, so the girth of $(G_{10}, m)$ equals 10.
\end{proof}
As we said before, the tree used in Theorem \ref{G10old} was also given in \cite{AEJ16}, and the order of the voltage graph $G_{10}$ is $26$, so the derived graph $(G_{10},m)$ gives us a family of $(\{3,m\};10)$-graphs with $24m+2$ vertices. 

Next, we describe a new voltage graph $H_{10}$ whose order is $22$ (see Figure \ref{fig:H10}). The derived graphs $(H_{10},m)$ produce a family of $(\{3,m\};10)$-graphs with $20m+2$ vertices, which obviously is an improvement. The derived graphs $(H_{10},m)$ are girth 10  for $m \geq 6$. 


\begin{figure}[htbp]
\begin{center}

\begin{tikzpicture}[
level distance=8mm,
level 1/.style={sibling distance=9 cm},
level 2/.style={sibling distance=7cm},
level 3/.style={sibling distance= 4cm},
level 4/.style={sibling distance= 2 cm},
vtx/.style={draw, circle, inner sep = 1 pt, font = \scriptsize},bigvtx/.style={draw, , inner sep = 1 pt, font = \scriptsize, minimum height = 10 pt},lbl/.style = {midway, fill = white, inner sep = 2 pt, font=\scriptsize}]

]


\tikzset{arc/.style = {red, thick, near start}}

\tikzset{edge from parent/.style={draw, edge from parent path={(\tikzparentnode) -- (\tikzchildnode)}}}

\node[bigvtx] (xx) {$x^{*}$}
	child[] {node[vtx,label = {[font = \scriptsize, label distance=-1pt]below:${x}$ }] (x) {}
		child[] {node[vtx,label =   {[font = \scriptsize, label distance=-1pt]left:${x_{0}}$} ] (x0) {}
			child{edge from parent[draw=none] node[draw=none]{}} 	
			child[]{node[vtx,label = {[font = \scriptsize, label distance=-1pt]left:${x_{01}}$}] (x00) {}
				child{node[vtx,label = {[font = \scriptsize, label distance=-1pt]left:${x_{010}}$}] (x000) {}}
				child{node[vtx,label = {[font = \scriptsize, label distance=-1pt]left:${x_{011}}$}] (x001) {}}
			}
		}
		child[] {node[vtx,label = {[font = \scriptsize, label distance=-1pt]right:$x_{1}$ }] (x1) {}
			child[]{node[vtx,label = {[font = \scriptsize, label distance=-1pt]left:${x_{10}}$}] (x10) {}
				child{node[draw = none
							] (x100) {}edge from parent[draw = none]}
				child{node[vtx,label = {[font = \scriptsize, label distance=-1pt]left:${x_{101}}$}] (x101) {}}
			}
			child[]{node[vtx,label = {[font = \scriptsize, label distance=-1pt]left:${x_{11}}$}] (x11) {}
				child{node[vtx,label = {[font = \scriptsize, label distance=-1pt]left:${x_{110}}$}] (x110) {}}
				child{node[draw = none
							] (x111) {} edge from parent[draw = none]}
			}
			}		
		};

\begin{scope}[grow = up, yshift = -9cm]		
\node[bigvtx] (yy) {$y^{*}$}
	child[] {node[vtx,label = {[font = \scriptsize, label distance=-1pt]above:${y}$} ] (y) {}
		child[] {node[vtx,label = {[font = \scriptsize, label distance=-1pt]right:$y_{1}$ }] (y1) {}
			child[]{node[vtx,label = {[font = \scriptsize, label distance=-1pt]left:${y_{11}}$}] (y11) {}
				child{node[ draw = none
								] (y111) {} edge from parent[draw = none]}
				child{node[vtx,label ={[font = \scriptsize, label distance=-1pt] left:${y_{110}}$}] (y110) {}}
				}
			child[]{node[vtx,label = {[font = \scriptsize, label distance=-1pt]left:${y_{10}}$}] (y10) {}
				child{node[vtx,label = {[font = \scriptsize, label distance=-1pt]left:${y_{101}}$}] (y101) {}}
				child{node[draw = none, 
								] (y100) {}edge from parent[draw = none]}
				}
			}	
	child[] {node[vtx,label =  {[font = \scriptsize, label distance=-1pt] left:${y_{0}}$ }] (y0) {}
			child[]{node[vtx,label = {[font = \scriptsize, label distance=-1pt]left:${y_{01}}$}] (y00) {}
				child{node[vtx,label = {[font = \scriptsize, label distance=-1pt]left:${y_{011}}$}] (y001) {}}
			child{node[vtx,label = {[font = \scriptsize, label distance=-1pt]left:${y_{010}}$}] (y000) {}}
			}
			child{ node[draw=none]{} edge from parent[draw=none]}
		}	
		};
\end{scope}	



\draw[ arc, black, -latex] (x0) to[bend right = 20]  (y0);
\draw[arc,  -latex, cyan] (x10) to[bend right = 25] node[lbl, pos = .2] {1} 
		(y10);
\draw[arc,  -latex, cyan] (x11) to[bend left = 25] node[lbl] {2} 
		(y11);
\draw[arc,  -latex] (x000) -- node[lbl]{2} 
	(y000);
\draw[arc,  -latex] (x000) -- node[lbl, pos = .2]{1} 
	(y110);
\draw[arc,  -latex] (x001) -- node[lbl, pos = .3]{2}
	(y101);
\draw[arc,  -latex] (x001) -- node[lbl, pos = .45]{3} 
	(y001);
\draw[arc,  -latex] (x110) -- node[lbl] {3} 
	(y110);
\draw[arc, -latex] (x110) -- node[lbl] {1}
	(y001);
\draw[arc,  -latex] (x101) -- node[lbl] {5} 
	(y000);
\draw[arc,  -latex] (x101) -- node[lbl]{3} 
	(y101);
	
	\begin{scope}[on background layer]
\draw[line width = 5 pt, opacity = .5, yellow] (x10.center) to[bend right = 25] (y10.center) -- (y101) -- (x101)  to[bend left = 0] (x10.center);
\draw[line width = 5 pt, opacity = .5, yellow] (x11.center) to[bend left = 25] (y11.center) -- (y110) -- (x110)  to[bend left = 0] (x11.center);
\end{scope}
	
\node[below right = of y1] {$\mathbb{Z}_{m}$};
\end{tikzpicture}

\caption{The voltage graph $H_{10}$, with voltage group $\mathbb{Z}_{m}$. Unlabeled edges all have voltage assignment 0, and the boxed vertices \fbox{$x^{*}$} and \fbox{$y^{*}$} are pinned vertices. Lifts of this graph are biregular graphs with two vertices of degree $m$ and $20m$ vertices of degree 3, which are all of girth 10 for $m \geq 6$. The two 4-cycles in the graph are highlighted.}
\label{fig:H10}
\end{center}
\end{figure}
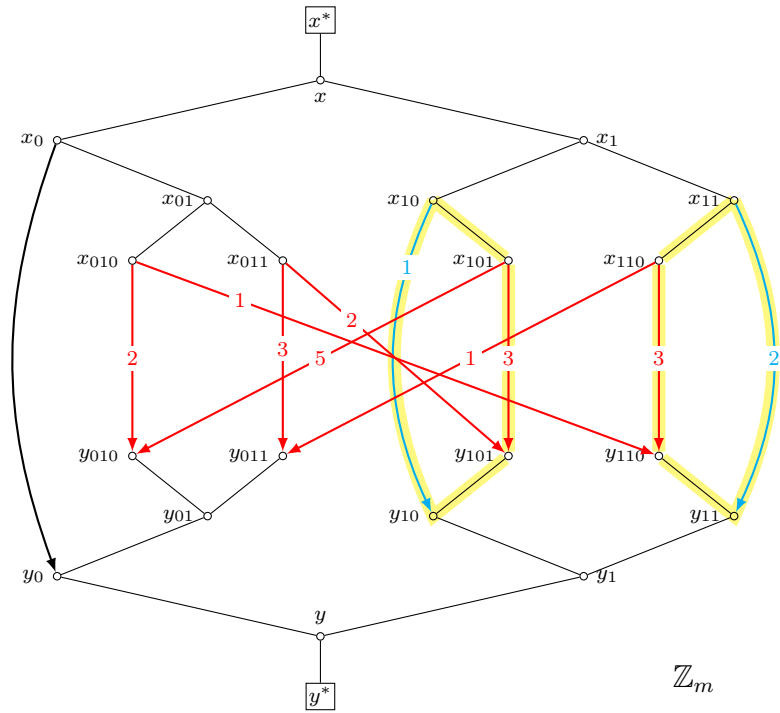

The graph $H_{10}$ is formed by pruning the tree $T_{10}$, by deleting the leaves $x_{101}$ and $y_{101}$ and their incident edges, and then adding an arc between $x_{10}$ and $y_{10}$, in addition to arcs between the remaining leaves, as in the construction of $G_{10}$.
 
\begin{theorem} 
The graph $(H_{10}; m)$ formed as a $\mathbb{Z}_{m}$ lift of the voltage graph shown in Figure \ref{fig:H10} has girth 10 for $m \geq 6$. Moreover, the graphs $(H_{10}, m)$ given us $(\{3,m\};10)$-graphs with $2$ vertices of degree $m$ and $20m$ vertices of degree $3$.
\end{theorem}
\begin{proof}
By construction, a shortest path between the pinned vertices lifts to a 10-cycle, so the girth of $(H_{10}, m)$ is at most 10. As before, to show that the girth is equal to 10 for $m \geq 6$, we analyze short cycles, short non-cycle closed walks, and short lollipop walks in the voltage graph, and argue that none of these lift to cycles in the derived graph.

Again using Mathematica \cite{Mathematica}, we considered all possible circuits in the voltage graph with doubled, oriented edges. There are 94 cycles possible: two disjoint 4-cycles (which can be oriented in either direction), highlighted in Figure \ref{fig:H10}, and a collection of 90 6- or 8-cycles. (Since $H$ is bipartite, there are no odd cycles, and all cycles of length 2 are reversing walks.) Summing the voltages along each of the cycles resulted in all of the voltage sums lying in the set $\{-5, -4, -3, -2, -1, 1, 2, 3, 4, 5\}$. Thus, for $m \geq 6$, none of the walks formed by going once around a cycle has a voltage sum of 0, so none of these cycles directly lifts to a cycle in $(H_{10}, m)$. (It is perhaps interesting to note that the red edges, oriented appropriately, form one of the 8-cycles, but the sum of the voltages along that 8-cycle, along with the sums of the voltages along the other 8-cycles, does not sum to $0 \bmod m$ for $m \geq 6$.)

Next, we looked at the 4-cycles individually. The 4-cycle $(x_{101}, y_{101}, y_{10}, x_{10}, x_{101})$ has voltage sum  2 (or -2, depending on orientation), so twice around that 4-cycle lifts to an 8-cycle when $m = 4$. The other 4-cycle sums to 1 (or -1) and so only lifts to larger cycles. (Since twice around a 6-cycle would lift to a 12-cycle, there is no need to consider larger cycles than 4 in this case.)

After that, we considered non-cycle closed walks which might lift to cycles, described in Observation \ref{lem:nonCycleWalks}. Since the two 4-cycles are disjoint, there are no 8-cycles that could be formed by traversing around those cycles in that way. All other joined cycles form walks that are at least length 10.

Finally, we considered lollipop walks. Inspection of $H_{10}$ shows that the shortest lollipop walks consist of paths of length 3 joined to a cycle of length 4 (e.g., $(x^{*},x,x_{1}, x_{10})$ joined to the cycle $(x_{10}, x_{101}, y_{101}, y_{10}, x_{10})$, which lifts to a cycle of length $2\cdot 3 + 4 = 10$, or a path of length 2 joined to a cycle of length 6 (e.g., $(x^{*},x, x_{0})$ joined to $(x_{0}, x_{00},x_{000}, y_{000}, y_{00}, y_{0}, x_{0})$ which lifts to a cycle of length $2 \cdot 2 + 6 = 10$.

\end{proof}

Notice that the voltage assignment of $G_{10}$ produces graphs of girth 10 for $m\geq 4$ except $m=6$, and for the voltage graph $H_{10}$ the voltage assignment produces graphs of girth 10 for $m\geq 6$. Thus, the only examples we know for $m = 4,5$ use $G_{10}$. For $m \geq 6$, while both $(G_{10}, m)$ and $(H_{10}, m)$ are girth 10, $(H_{10}, m)$ has fewer vertices than $(G_{10}, m)$. It is possible that other choices of arcs or other choices of labels could produce a variant of $H_{10}$ that has girth 10 for smaller values of $m$, but so far, such construction has eluded us.
 It is also interesting to note that, as we said before, graphs with the same order of $(G_{10}, m)$ and with $T_{10}$ as based were also given in  \cite{AEJ16}, but in that paper, the authors give general constructions for $m$ that is very large in relation to $3$. Also in that paper, they constructed a $(\{3,4\};10)$-graph with $82$ vertices, which is exactly the order of a graph $(H_{10}, 4)$. However, we have not yet found a voltage assignment for $H_{10}$ that produces a girth 10 graph for $m=4$. Moreover, the $(\{3,4\};10)$-graph constructed in \cite{AEJ16} has $4$ vertices of degree $4$ and $78$ vertices of degree $3$, so clearly it is not a graph in our family.

\section{Families of semicubic graphs of girth $4t$}\label{case4t}
In this section, as in the previous one, we will use a family of voltage graphs $\mathcal{G}_{4t}$ to construct a family of $(\{3;m\};4t)$-biregular graphs called  $(G_{4t}; m)$.
In this case, the voltage graph construction begins by constructing three pruned binary trees, 
each extended by connecting a pinned vertex to the root. 

We begin with constructing a new pruned binary tree $X'_{t}$, which has height $2t-1$, including the vertices $x^{*}$ and $x$. As in the previous section, we begin with a binary tree rooted at a vertex $x$, whose vertices are indexed by bit strings of length at most $2t-2$. We extend the tree upwards by a single pinned vertex $x^{*}$, which becomes the new root of the tree. We identify the vertex $x_{a}$ where $a = \underbrace{0\cdots0}_{t-1}$ which is at distance $t$ from the root $x^{*}$. We then delete all the children of $x_{a}$, that is, all the vertices in the binary tree indexed by bit strings that begin with strings of $t-1$ zeroes. Note that unlike the tree $X_{t}$ from the previous section, $X'_{t}$ has height $2t-1$ rather than $2t$, and by construction, the distance from $x^{*}$ to $x_{a}$ is $t$, and the distance from the level of $x_{a}$ to the level of the other leaves is $t-1$. Moreover, the vertex $x_{a}$ is itself a leaf, unlike in the construction of $X_{t}$, where it had degree 2. Notice that the bit strings of the leaves, except $x_a$, have length $2t-2$.


We then take the pruned tree $Y_{t-1}$ that was constructed in the previous section, which contains a vertex $y_{a}$ where $a = \underbrace{0\cdots0}_{t-2}$ (when $t = 2$ we define $y_{a}$ to be the vertex $y$, indexed by the empty bitstring). In this tree, the distance from $y^{*}$ to $y_{a}$ is $t-1$, and the distance from $y_{a}$ to the level of the leaves of the tree is also $t-1$. We define $Z_{t-1}$ identically, changing the name only to keep track of the second copy. Here the bit strings of the leaves have lengths $2t-3.$


To complete the construction of the tree $T_{4t}$, we join $X'_{t}$, $Y_{t-1}$ and $Z_{t-1}$ by adding edges $(x_{a}, y_{a})$ and $(x_{a}, z_{a})$ with label 0. An example for $t = 3$ is shown in Figure \ref{fig:T-4t}.

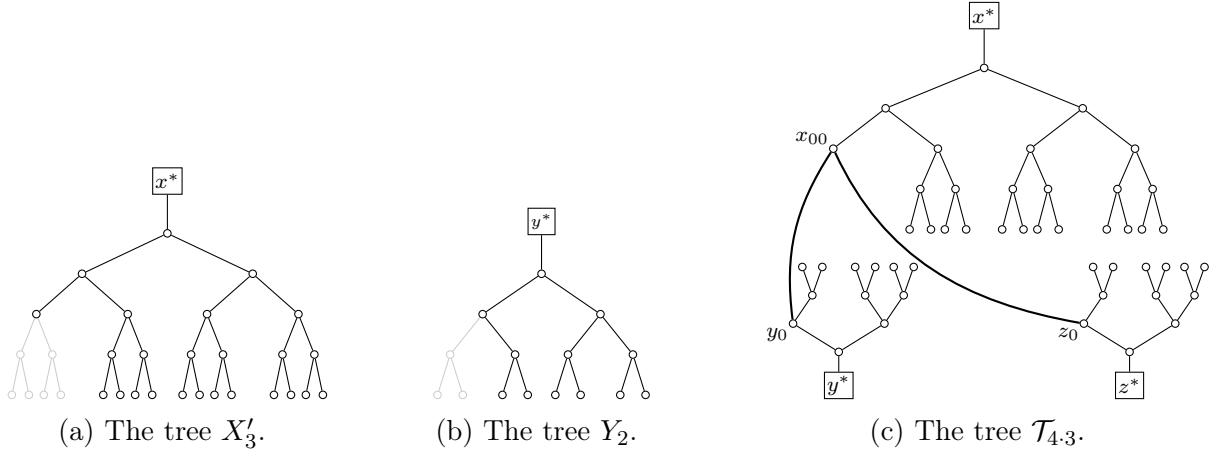
\begin{figure}[htbp]
\begin{center}
\subfloat[The tree $X'_{3}$.]{
\begin{tikzpicture}[vtx/.style={draw, circle, inner sep = 1 pt, font = \scriptsize},bigvtx/.style={draw, , inner sep = 1 pt, font = \scriptsize,minimum height = 10 pt},lbl/.style = {midway, fill = white, inner sep = 2 pt, font=\footnotesize}]
\pgfmathsetmacro{\xx}{30}
\pgfmathsetmacro{\yy}{13}
\pgfmathsetmacro{\r}{32}
\pgfmathsetmacro{\rr}{8}
\pgfmathsetmacro{\rrr}{2}
\node[bigvtx] (xx) {$x^{*}$};
\node[vtx, below = \yy pt of xx] (x){};
\node[vtx, below left = \yy pt and \xx pt of x] (x0) {};
\node[vtx, below right = \yy pt and \xx pt of x] (x1){}; 
\node[vtx, below left = \yy pt and \xx/\rrr pt  of x0,] (x00) {};
\node[vtx, below right = \yy pt and \xx/\rrr pt  of x0] (x01) {};
\node[vtx, below left = \yy pt and \xx/\rrr pt  of x1] (x10) {};
\node[vtx, below right = \yy pt and \xx/\rrr pt  of x1] (x11) {};

\node[vtx, below left = \yy pt and \xx/\rr pt  of x00, opacity = .2] (x000) {};
\node[vtx, below left = \yy pt and \xx/\rr pt  of x01] (x010) {};
\node[vtx, below left = \yy pt and \xx/\rr pt  of x10] (x100) {};
\node[vtx, below left = \yy pt and \xx/\rr pt  of x11] (x110) {};

\node[vtx, below right = \yy pt and \xx/\rr pt  of x00, opacity = .2] (x001) {};
\node[vtx, below right = \yy pt and \xx/\rr pt  of x01] (x011) {};
\node[vtx, below right = \yy pt and \xx/\rr pt  of x10] (x101) {};
\node[vtx, below right = \yy pt and \xx/\rr pt  of x11] (x111) {};

\node[vtx, below left = \yy pt and \xx/\r pt  of x000, opacity = .2] (x0000) {};%
\node[vtx, below left = \yy pt and \xx/\r pt  of x010] (x0100) {};
\node[vtx, below left = \yy pt and \xx/\r pt  of x100] (x1000) {};
\node[vtx, below left = \yy pt and \xx/\r pt  of x110] (x1100) {};
\node[vtx, below left = \yy pt and \xx/\r pt  of x001, opacity = .2] (x0010) {};%
\node[vtx, below left = \yy pt and \xx/\r pt  of x011] (x0110) {};
\node[vtx, below left = \yy pt and \xx/\r pt  of x101] (x1010) {};
\node[vtx, below left = \yy pt and \xx/\r pt  of x111] (x1110) {};
\node[vtx, below right = \yy pt and \xx/\r pt  of x000, opacity = .2] (x0001) {};%
\node[vtx, below right = \yy pt and \xx/\r pt  of x010] (x0101) {};
\node[vtx, below right = \yy pt and \xx/\r pt  of x100] (x1001) {};
\node[vtx, below right = \yy pt and \xx/\r pt  of x110] (x1101) {};
\node[vtx, below right = \yy pt and \xx/\r pt  of x001, opacity = .2] (x0011) {};%
\node[vtx, below right = \yy pt and \xx/\r pt  of x011] (x0111) {};
\node[vtx, below right = \yy pt and \xx/\r pt  of x101] (x1011) {};
\node[vtx, below right = \yy pt and \xx/\r pt  of x111] (x1111) {};

\draw (xx) -- (x);
\foreach \i in {0,1}{
\draw (x) -- (x\i);}
\foreach \j in {0,1}{
\draw (x1) -- (x1\j);
\draw (x10) -- (x10\j);
\draw (x01) -- (x01\j);
\draw (x11) -- (x11\j);

\draw[] (x010) -- (x010\j);
\draw[] (x011) -- (x011\j);
\draw[] (x100) -- (x100\j);
\draw[] (x101) -- (x101\j);
\draw[] (x110) -- (x110\j);
\draw[] (x111) -- (x111\j);
}
\draw (x0) -- (x01);
\draw[] (x0)-- (x00);
\draw[opacity = .2] (x00)-- (x000);
\draw[opacity = .2] (x00)-- (x001);
\draw[opacity = .2] (x000) -- (x0000);
\draw[opacity = .2] (x001) -- (x0010);
\draw[opacity = .2] (x000) -- (x0001);
\draw[opacity = .2] (x001) -- (x0011);


\end{tikzpicture}
}
\hspace{1cm}
\subfloat[The tree $Y_{2}$.]{
\begin{tikzpicture}[vtx/.style={draw, circle, inner sep = 1 pt, font = \tiny},bigvtx/.style={draw, , inner sep = 1 pt, font = \tiny,minimum height = 10 pt}, lbl/.style = {midway, fill = white, inner sep = 2 pt, font=\footnotesize}]
\pgfmathsetmacro{\xx}{20}
\pgfmathsetmacro{\yy}{13}
\pgfmathsetmacro{\r}{8}

\node[bigvtx] (yy) {$y^{*}$};
\node[vtx, below = \yy pt of yy] (y){};
\node[vtx, below left = \yy pt and \xx pt of y] (y0) {};
\node[vtx, below right = \yy pt and \xx pt of y] (y1){}; 
\node[vtx, below left = \yy pt and \xx/2 pt  of y0, opacity = .2] (y00) {};
\node[vtx, below right = \yy pt and \xx/2 pt  of y0] (y01) {};
\node[vtx, below left = \yy pt and \xx/2 pt  of y1] (y10) {};
\node[vtx, below right = \yy pt and \xx/2 pt  of y1] (y11) {};

\node[vtx, below left = \yy pt and \xx/\r pt  of y00, opacity = .2] (y000) {};
\node[vtx, below left = \yy pt and \xx/\r pt  of y01] (y010) {};
\node[vtx, below left = \yy pt and \xx/\r pt  of y10] (y100) {};
\node[vtx, below left = \yy pt and \xx/\r pt  of y11] (y110) {};

\node[vtx, below right = \yy pt and \xx/\r pt  of y00, opacity = .2] (y001) {};
\node[vtx, below right = \yy pt and \xx/\r pt  of y01] (y011) {};
\node[vtx, below right = \yy pt and \xx/\r pt  of y10] (y101) {};
\node[vtx, below right = \yy pt and \xx/\r pt  of y11] (y111) {};

\draw (yy) -- (y);
\foreach \i in {0,1}{
\draw (y) -- (y\i);}
\foreach \j in {0,1}{
\draw (y1) -- (y1\j);
\draw (y10) -- (y10\j);
\draw (y01) -- (y01\j);
\draw (y11) -- (y11\j);
}
\draw (y0) -- (y01);
\draw[opacity = .2] (y0)-- (y00);
\draw[opacity = .2] (y00)-- (y000);
\draw[opacity = .2] (y00)-- (y001);


\end{tikzpicture}
}
\hspace{1cm}
\subfloat[The tree $\mc{T}_{4\cdot 3}$.]{
\begin{tikzpicture}[vtx/.style={draw, circle, inner sep = 1 pt, font = \scriptsize},bigvtx/.style={draw, , inner sep = 1 pt, font = \scriptsize, minimum height = 10 pt},lbl/.style = {midway, fill = white, inner sep = 2 pt, font=\footnotesize}]
\pgfmathsetmacro{\xx}{35}
\pgfmathsetmacro{\yy}{13}
\pgfmathsetmacro{\r}{20}
\pgfmathsetmacro{\rr}{8}
\pgfmathsetmacro{\rrr}{2}
\node[bigvtx] (xx) {$x^{*}$};
\node[vtx, below = \yy pt of xx] (x){};
\node[vtx, below left = \yy pt and \xx pt of x] (x0) {};
\node[vtx, below right = \yy pt and \xx pt of x] (x1){}; 
\node[vtx, below left = \yy pt and \xx/\rrr pt  of x0,label = {[font = \scriptsize, label distance=-5pt]above left:{$x_{00}$}}] (x00) {};
\node[vtx, below right = \yy pt and \xx/\rrr pt  of x0] (x01) {};
\node[vtx, below left = \yy pt and \xx/\rrr pt  of x1] (x10) {};
\node[vtx, below right = \yy pt and \xx/\rrr pt  of x1] (x11) {};

\node[vtx, below left = \yy pt and \xx/\rr pt  of x01] (x010) {};
\node[vtx, below left = \yy pt and \xx/\rr pt  of x10] (x100) {};
\node[vtx, below left = \yy pt and \xx/\rr pt  of x11] (x110) {};

\node[vtx, below right = \yy pt and \xx/\rr pt  of x01] (x011) {};
\node[vtx, below right = \yy pt and \xx/\rr pt  of x10] (x101) {};
\node[vtx, below right = \yy pt and \xx/\rr pt  of x11] (x111) {};

\node[vtx, below left = \yy pt and \xx/\r pt  of x010] (x0100) {};
\node[vtx, below left = \yy pt and \xx/\r pt  of x100] (x1000) {};
\node[vtx, below left = \yy pt and \xx/\r pt  of x110] (x1100) {};
\node[vtx, below left = \yy pt and \xx/\r pt  of x011] (x0110) {};
\node[vtx, below left = \yy pt and \xx/\r pt  of x101] (x1010) {};
\node[vtx, below left = \yy pt and \xx/\r pt  of x111] (x1110) {};
\node[vtx, below right = \yy pt and \xx/\r pt  of x010] (x0101) {};
\node[vtx, below right = \yy pt and \xx/\r pt  of x100] (x1001) {};
\node[vtx, below right = \yy pt and \xx/\r pt  of x110] (x1101) {};
\node[vtx, below right = \yy pt and \xx/\r pt  of x011] (x0111) {};
\node[vtx, below right = \yy pt and \xx/\r pt  of x101] (x1011) {};
\node[vtx, below right = \yy pt and \xx/\r pt  of x111] (x1111) {};

\draw (xx) -- (x);
\foreach \i in {0,1}{
\draw (x) -- (x\i);}
\foreach \j in {0,1}{
\draw (x1) -- (x1\j);
\draw (x10) -- (x10\j);
\draw (x01) -- (x01\j);
\draw (x11) -- (x11\j);

\draw[] (x010) -- (x010\j);
\draw[] (x011) -- (x011\j);
\draw[] (x100) -- (x100\j);
\draw[] (x101) -- (x101\j);
\draw[] (x110) -- (x110\j);
\draw[] (x111) -- (x111\j);
}
\draw (x0) -- (x01);
\draw[] (x0)-- (x00);

\pgfmathsetmacro{\sh}{55}
\pgfmathsetmacro{\ysh}{-140}

\begin{scope}[yshift = \ysh, xshift = -\sh]
\pgfmathsetmacro{\xx}{15}
\pgfmathsetmacro{\yy}{-13}
\pgfmathsetmacro{\rr}{3}
\pgfmathsetmacro{\rrr}{10}

\node[bigvtx, ] (yy) {$y^{*}$};
\node[vtx, below = 1.5*\yy pt of yy] (y){};
\node[vtx, below left = \yy pt and \xx pt of y, label = {[font = \scriptsize, label distance=-5pt]below left:{$y_{0}$}}] (y0) {};
\node[vtx, below right = \yy pt and \xx pt of y] (y1){}; 
\node[vtx, below right = \yy pt and \xx/\rr pt  of y0] (y01) {};
\node[vtx, below left = \yy pt and \xx/\rr pt  of y1] (y10) {};
\node[vtx, below right = \yy pt and \xx/\rr pt  of y1] (y11) {};

\node[vtx, below left = \yy pt and \xx/\rrr pt  of y01] (y010) {};
\node[vtx, below left = \yy pt and \xx/\rrr pt  of y10] (y100) {};
\node[vtx, below left = \yy pt and \xx/\rrr pt  of y11] (y110) {};

\node[vtx, below right = \yy pt and \xx/\rrr pt  of y01] (y011) {};
\node[vtx, below right = \yy pt and \xx/\rrr pt  of y10] (y101) {};
\node[vtx, below right = \yy pt and \xx/\rrr pt  of y11] (y111) {};

\draw (yy) -- (y);
\foreach \i in {0,1}{
\draw (y) -- (y\i);}
\foreach \j in {0,1}{
\draw (y1) -- (y1\j);
\draw (y10) -- (y10\j);
\draw (y01) -- (y01\j);
\draw (y11) -- (y11\j);
}
\draw (y0) -- (y01);

\end{scope}

\begin{scope}[yshift = \ysh, xshift = \sh]
\pgfmathsetmacro{\xx}{15}
\pgfmathsetmacro{\yy}{-13}
\pgfmathsetmacro{\rr}{3}
\pgfmathsetmacro{\rrr}{10}

\node[bigvtx, ] (zz) {$z^{*}$};
\node[vtx, below = 1.5*\yy pt of zz] (z){};
\node[vtx, below left = \yy pt and \xx pt of z, label = {[font = \scriptsize, label distance=-5pt]below left:{$z_{0}$}}] (z0) {};
\node[vtx, below right = \yy pt and \xx pt of z] (z1){}; 
\node[vtx, below right = \yy pt and \xx/\rr pt  of z0] (z01) {};
\node[vtx, below left = \yy pt and \xx/\rr pt  of z1] (z10) {};
\node[vtx, below right = \yy pt and \xx/\rr pt  of z1] (z11) {};

\node[vtx, below left = \yy pt and \xx/\rrr pt  of z01] (z010) {};
\node[vtx, below left = \yy pt and \xx/\rrr pt  of z10] (z100) {};
\node[vtx, below left = \yy pt and \xx/\rrr pt  of z11] (z110) {};

\node[vtx, below right = \yy pt and \xx/\rrr pt  of z01] (z011) {};
\node[vtx, below right = \yy pt and \xx/\rrr pt  of z10] (z101) {};
\node[vtx, below right = \yy pt and \xx/\rrr pt  of z11] (z111) {};

\draw (zz) -- (z);
\foreach \i in {0,1}{
\draw (z) -- (z\i);}
\foreach \j in {0,1}{
\draw (z1) -- (z1\j);
\draw (z10) -- (z10\j);
\draw (z01) -- (z01\j);
\draw (z11) -- (z11\j);
}
\draw (z0) -- (z01);

\end{scope}

\draw[ thick] (x00) to[bend right = 20] (y0);
\draw[ thick] (x00) to[out = -65, in = 180-10] (z0);
\end{tikzpicture}

}

\caption{The trees $X'_{t}$, $Y_{t-1}$, and tree $T_{4t}$, for $t = 3$. (The light gray subgraphs have been pruned.)}
\label{fig:T-4t}
\end{center}
\end{figure}


The total number of vertices in $T_{4t}$ is given by the following analysis: the number of vertices of $X'_{t}$, excluding the vertex $x^*$, is: $\sum_{i=0}^{2t-2}2^i -\sum_{i=1}^{t-1}2^i =1+\sum_{i=t}^{2t-2}2^i = 1+2^{2t-2}+\sum_{i=t}^{2t-3}2^i$; the number of vertices of $Y_{t-1}$ (and also $Z_{t-1}$) is: $\sum_{i=0}^{2t-3}2^i-\sum_{i=0}^{t-2}2^i=\sum_{i=t-1}^{2t-3}2^i.$

Thus, the total number of vertices of the derived graph is $3+m\{1+2^{2t-2}+2^t+3\sum_{i=t}^{2t-3}2^i\}$, where 3 of the vertices are the pinned vertices $x^{*}$, $y^{*}$, $z^{*}$. 


As in the previous section, an element of the family $\mathcal{G}_{4t}$ is any voltage graph formed by adding arcs and voltage assignments to $T_{4t}$ so that all vertices other than the pinned vertices have degree 3; in this paper, we restrict our additions so that each leaf in $X'_{t}$ has one arc going to some leaf in $Y_{t-1}$ and one arc going to some leaf in $Z_{t-1}$. As usual, $(G_{4t}, m)$ is the $\mathbb{Z}_{m}$ lift of  voltage graph $G_{4t}$.


As in the previous section, we prove the following results.\\

\begin{lemma}\label{txygirth8}
The graph $(T_{4t}; m)$ has girth $4t$.
\end{lemma}
\begin{proof}
By construction, since the distance in $X'_{t}$ between $x^{*}$ and $x_a$ is $t$, and the distance in $Y_{t-1}$ between $y^{*}$ and $y_a$ is $t-1$ (analogously in $Z_{t-1}$ between  $z^{*}$ and $z_a$), there exists a minimal path $P=(x^{*},\ldots,x_a,y_a,\ldots,y^{*})$ of length $2t$ between the two pinned vertices $x^{*}$ and $y^{*}$ (and similarly a path between $x^{*}$ and $z^{*}$). Also, there exists another minimal path $P'=(y^{*},\ldots,y_a,x_a,z_a,\ldots,z^{*})$ between $y^{*}$ and $z^{*}$ of length $2(t-1)+2=2t$. 

By Lemma \ref{lemma:pinnedwalks}, the paths $P$ and $P'$ lift to cycles of length $4t$ in $(T_{4t}, m)$.
no other shorter paths in $T_{4t}$ lift to cycles, since $T_{4t}$ is a tree. In consequence the girth of $(T_{4t}; m)$ is $4t$.
\end{proof}
As $(T_{4t}; m)$ is a subgraph of $(G_{4t}; m)$, as in the previous section, we have the following result: 

\begin{lemma}
The graph $(G_{4t}; m)$ has girth at most $4t$.
\end{lemma}

Now, we will prove that: 
\begin{lemma}\label{bipartite_}
The family $\mathcal{G}_{4t}$ is bipartite.
\end{lemma}

\begin{proof}
As in the proof of Lemma \ref{bipartite}, let $\mathcal{X}_j$ be the set of vertices of the pruned binary tree $X'_{t}$ contained in the voltage graph $\mc{G}_{4t}$ at distance $j$ from $x^{*}$ and the same for the trees $Y_{t-1}$ and $Z_{t-1}$. Observe that $B_{\mathcal{X}^{'}}=\{\mathcal{X}_0, \mathcal{X}_2, ..., \mathcal{X}_{2t-2}\}$ and $B_{\mathcal{X}^{''}}=\{\mathcal{X}_1, \mathcal{X}_3, ..., \mathcal{X}_{2t-1}\}$ are two disjoint subsets of vertices of $X'_{t}$ whose union is $V(X'_{t})$. 

Similarly,  $B_{\mathcal{Y}^{'}}=\{\mathcal{Y}_0, \mathcal{Y}_2, ..., \mathcal{Y}_{2t-2}\}$ and $B_{\mathcal{Y}^{''}}=\{\mathcal{Y}_1, \mathcal{Y}_3, ..., \mathcal{Y}_{2t-3}\}$ are two disjoint subsets of the vertices of $Y_{t-1}$ whose union is $V(Y_{t-1})$ and analogously for $Z_{t-1}$. Thus, the bipartite classes of $\mc{G}_{4t}$ are $\mathcal{B}_1=B_{\mathcal{X}^{'}}\cup B_{\mathcal{Y}^{'}}\cup B_{\mathcal{Z}^{'}}$ and $\mathcal{B}_2=B_{\mathcal{X}^{''}}\cup B_{\mathcal{Y}^{''}}\cup B_{\mathcal{Z}^{''}}$. 

\end{proof}

\subsubsection{A family of semicubic graphs of girth $8$}
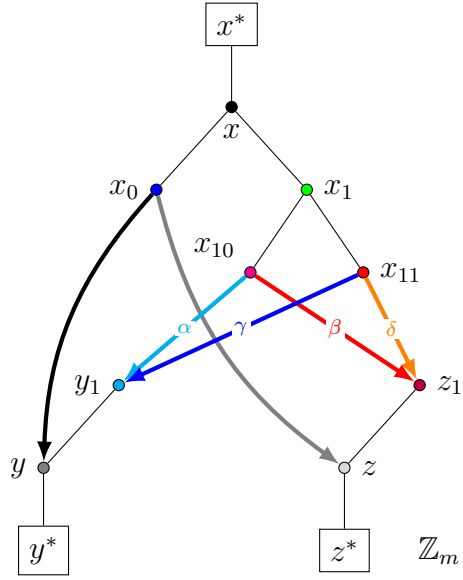
\begin{figure}[h!]
\centering

\begin{tikzpicture}[
level distance=11mm,
level 1/.style={sibling distance=3 cm},
level 2/.style={sibling distance=2cm},
level 3/.style={sibling distance= 1.5cm},
level 4/.style={sibling distance= 1 cm}]
\tikzset{vtx/.style={draw, fill = black, circle, align = center, minimum size=1.5 pt, inner sep = 1.5 pt}}
\tikzset{lbl/.style={font = \scriptsize, fill = white, inner sep = 1 pt}}

\tikzset{bigvtx/.style={draw, align = center, minimum size=1 pt, inner sep = 4 pt}}

\tikzset{arc/.style = {red, ultra thick, near start}}

\tikzset{edge from parent/.style={draw, edge from parent path={(\tikzparentnode) -- (\tikzchildnode)}}}

\node[bigvtx] (xx) {$x^{*}$}
	child[] {node[vtx,label = below:${x}$ 
	] (x) {}
		child[] {node[vtx,fill = blue,label =   left:${x_{0}}$ 
		] (x0) {}
			}	
		child[] {node[vtx,fill = green,label = right:$x_{1}$ 
		] (x1) {}
			child[]{node[vtx,fill= magenta,label = {above left=1mm}:${x_{10}}$] (x10) {}
			}
			child[]{node[vtx,fill = red,label = right:${x_{11}}$] (x11) {}
			}
			}		
		};
		
\begin{scope}[grow = up, yshift = -7 cm, xshift = -2.5 cm] 
\node[bigvtx] (yy) {$y^{*}$}
	child[] {node[vtx,fill = gray,label = left:${y}$ 
	] (y) {}
		child[] {node[vtx,fill = cyan,label =   left:${y_{1}}$ ] (y0) {}			}	
			child{node[draw = none](y1) {}edge from parent[draw = none]}
		};
		\end{scope}
		
		\begin{scope}[grow = up, yshift = -7 cm, xshift=1.5cm] 
\node[bigvtx] (zz) {$z^{*}$}
	child[] {node[vtx,fill = gray!30,label = right:${z}$ ] (z) {}
		child[] {node[vtx,fill = purple,label =   right:${z_{1}}$ ] (z0) {}
			}	
			child{node[draw = none](z1) {}edge from parent[draw = none]}
		};
		\end{scope}

\draw[ arc,black, ->] (x0) to[bend right = 20]  (y);
\draw[ arc,gray, ->] (x0) to[bend right = 20]  (z);
\draw[arc, ->, cyan] (x10) to[bend right = 0] node[lbl, midway] {$\alpha$} 
		(y0);
\draw[arc, ->, red] (x10) to[bend left = 0] node[lbl, pos = .5] {$\beta$} 
	(z0);
	\draw[arc, ->, blue] (x11) to[bend left = 0] node[lbl, pos = .5] {$\gamma$} 
	(y0);
		\draw[arc, ->, orange] (x11) to[bend left = 0] node[lbl, midway] {$\delta$} 
	(z0);
\node[right = 5mm of zz] {$\mathbb{Z}_{m}$};
\end{tikzpicture}

\caption{An assignment of arcs and voltages $\alpha$, $\beta$, $\gamma$, $\delta$ to $\mc{G}_8$ that corresponds to a family of semi-regular graphs of girth $8$, when $\alpha$, $\beta$, $\gamma$, $\delta$ are constrained as in Theorem \ref{thm:g8}}
\label{g8voltage}
\end{figure}

\begin{figure}[h!]
\centering
\subfloat[$(G_{8}; 3)$ with $\alpha = \delta=1$ and $\beta = \gamma = 2$...]{
\begin{tikzpicture}[
level distance=7mm,
level 1/.style={sibling distance=1 cm},
level 2/.style={sibling distance=1cm},
level 3/.style={sibling distance= 1cm},
level 4/.style={sibling distance= 1 cm}, scale = .65, every node/.append style={transform shape}]
\tikzset{vtx/.style={draw, fill = black, circle, align = center, minimum size=2 pt, inner sep = 2 pt}}
\tikzset{lbl/.style={font = \tiny, fill = white, inner sep = 1 pt}}

\tikzset{bigvtx/.style={draw, align = center, minimum size=1 pt, inner sep = 4 pt}}

\tikzset{arc/.style = {red, ultra thick, near start}}

\tikzset{edge from parent/.style={draw, edge from parent path={(\tikzparentnode) -- (\tikzchildnode)}}}

\node[bigvtx] (xx) at (5.5,1.5) {$x^{*}$};

\foreach \j in {0,1,2}{
\pgfmathparse{\j*3.5+2}
\begin{scope}[xshift = \pgfmathresult cm,]
\node[vtx,label = above left:${x^{\j}}$ ] (x\j) {}{
		child[] {node[vtx,fill = blue,label =   left:${x_{0}^{\j}}$ 
		] (x0\j) {}
			}	
		child[] {node[vtx,fill = green,label = right:$x_{1}^{\j}$ 
		] (x1\j) {}
			child[]{node[vtx,fill= magenta,label = {above left=1mm}:${x_{10}^{\j}}$] (x10\j) {}
			}
			child[]{node[vtx,fill = red,label = right:${x_{11}^{\j}}$] (x11\j) {}
			}
			}		
		};
		\end{scope}
		}

\node[bigvtx] (yy) at (2.5, -6) {$y^{*}$};

\foreach \j in {0,1,2}{
\pgfmathparse{ 1 + 1.1*\j}
\begin{scope}[grow = up, yshift = -5 cm, xshift = \pgfmathresult cm] 
\node[vtx,fill = gray,label = left:${y^{\j}}$ ] (y\j) {}{
		child[] {node[vtx,fill = cyan,label =   left:${y_{1}^{\j}}$ ] (y0\j) {}			}	
			child{node[draw = none](y1\j) {}edge from parent[draw = none]}
		};
		\end{scope}
		}

\node[bigvtx] (zz) at (9, -6) {$z^{*}$};

\foreach \j in {0,1,2}{		
\pgfmathparse{1.1*\j +8}
\begin{scope}[grow = up, yshift = -5 cm, xshift = \pgfmathresult cm] 	

\node[vtx,fill = gray!30,label = right:${z^{\j}}$ ] (z\j) {}{
		child[] {node[vtx,fill = purple,label =   right:${z_{1}^{\j}}$ ] (z0\j) {}
			}	
			child{node[draw = none](z1\j) {}edge from parent[draw = none]}
		};
		\end{scope}
		}

\foreach \j in {0,1,2}{
\draw (xx) -- (x\j);
\draw (yy) -- (y\j);
\draw (zz) -- (z\j);
\draw[ arc,black, ] (x0\j) to[bend right = 20]  (y\j);
\draw[ arc,gray, ] (x0\j) to[bend right = 20]  (z\j);
\draw[arc, , cyan] let \n1 = {int(mod(\j+1,3))} in (x10\j) to[bend right = 0] 
		(y0\n1);
\draw[arc, , red]let \n1 = {int(mod(\j+2,3))} in (x10\j) to[bend left = 0] 
	(z0\n1);
	\draw[arc, , blue] let \n1 = {int(mod(\j+2,3))} in (x11\j) to[bend left = 0] 
	(y0\n1);
		\draw[arc, , orange] let \n1 = {int(mod(\j+1,3))} in (x11\j) to[bend left = 0] 
	(z0\n1);
	}

\end{tikzpicture}
}
\subfloat[...is isomorphic to the Tutte 8-cage]{
\begin{tikzpicture}[scale = .6,every node/.append style={transform shape}]

\tikzset{vtx/.style={draw, fill = black, circle, align = center, minimum size=2 pt, inner sep = 2 pt}}
\tikzset{lbl/.style={font = \tiny, fill = white, inner sep = 1 pt}}

\tikzset{bigvtx/.style={draw, align = center, minimum size=1 pt, inner sep = 1 pt}}
\tikzset{arc/.style = {red, ultra thick, near start}}

\def\rr{4}
\node[bigvtx,] (xx) at (360*0/30:\rr) {$x^{*}$};
\node[vtx,label =  {360*1/30, label distance = -6pt}:${x^{0}}$ ] (x0)  at (360*1/30:\rr) {};
\node[vtx,fill = blue,label =  {360*2/30, label distance = -6pt}:${x_{0}^{0}}$]  (x00) at (360*2/30:\rr) {};
\node[vtx,fill = gray!30,label = {360*3/30, label distance = -6pt}:${z^{0}}$ ]  (z0) at (360*3/30:\rr) {};
\node[vtx,fill = purple,label =   {360*4/30, label distance = -6pt}:${z_{1}^{0}}$ ] (z00)at (360*4/30:\rr) {};
\node[vtx,fill= magenta,label = {360*5/30, label distance = -6pt}:${x_{10}^{1}}$] (x101) at (360*5/30:\rr) {};
\node[vtx,fill = cyan,label =   {360*6/30, label distance = -6pt}:${y_{1}^{2}}$ ] (y02)at (360*6/30:\rr) {};	
\node[vtx,fill = gray,label = {360*7/30, label distance = -6pt}:${y^{2}}$ ] (y2) at (360*7/30:\rr){};
\node[vtx,fill = blue,label =   {360*9/30, label distance = -6pt}:${x_{0}^{2}}$] (x02)at (360*8/30:\rr) {};
\node[vtx,fill = gray!30,label = {360*10/30, label distance = -6pt}:${z^{2}}$ ] (z2)at (360*9/30:\rr) {};
\node[bigvtx] (zz) at (360*10/30:\rr){$z^{*}$};
\node[vtx,fill = gray!30,label = {360*11/30, label distance = -6pt}:${z^{1}}$ ] (z1)at (360*11/30:\rr) {};
\node[vtx,fill = blue,label =   {360*12/30, label distance = -6pt}:${x_{0}^{1}}$ ] (x01)at (360*12/30:\rr) {};
\node[vtx,label = {360*13/30, label distance = -6pt}:${x^{1}}$ ] (x1) at (360*13/30:\rr){};
\node[vtx,fill = green,label = {360*14/30, label distance = -5pt}:$x_{1}^{1}$ ] (x11) at (360*14/30:\rr) {};
\node[vtx,fill = red,label = {360*15/30, label distance = -5pt}:${x_{11}^{1}}$] (x111) at (360*15/30:\rr) {};
\node[vtx,fill = purple,label =   {360*16/30, label distance = -6pt}:${z_{1}^{2}}$ ] (z02) at (360*16/30:\rr){};
\node[vtx,fill= magenta,label = {360*17/30, label distance = -6pt}:${x_{10}^{0}}$] (x100)at (360*17/30:\rr) {};
\node[vtx,fill = green,label = {360*18/30, label distance = -6pt}:$x_{1}^{0}$ ] (x10) at (360*18/30:\rr){};
\node[vtx,fill = red,label = {360*19/30, label distance = -6pt}:${x_{11}^{0}}$] (x110) at (360*19/30:\rr){};
\node[vtx,fill = purple,label =   {360*20/30, label distance = -6pt}:${z_{1}^{1}}$ ] (z01) at (360*20/30:\rr){};
\node[vtx,fill= magenta,label = {360*21/30, label distance = -6pt}:${x_{10}^{2}}$] (x102) at (360*21/30:\rr) {};
\node[vtx,fill = cyan,label =   {360*22/30, label distance = -5pt}:${y_{1}^{0}}$ ] (y00)at (360*22/30:\rr) {};
\node[vtx,fill = gray,label = {360*23/30, label distance = -6pt}:${y^{1}}$ ] (y0) at (360*23/30:\rr){};
\node[bigvtx] (yy) at (360*24/30:\rr) {$y^{*}$};
\node[vtx,fill = gray,label = {360*25/30, label distance = -6pt}:${y^{1}}$ ] (y1) at (360*25/30:\rr) {};
\node[vtx,fill = cyan,label =   {360*26/30, label distance = -6pt}:${y_{1}^{0}}$ ] (y01)at (360*26/30:\rr) {};
\node[vtx,fill = red,label = {360*27/30, label distance = -6pt}:${x_{11}^{2}}$] (x112)at (360*27/30:\rr) {};
\node[vtx,fill = green,label = {360*28/30, label distance = -6pt}:$x_{1}^{2}$ ] (x12) at (360*28/30:\rr) {};
\node[vtx,label = {360*29/30, label distance = -6pt}:${x^{2}}$ ] (x2) at (360*29/30:\rr) {};

\foreach \j in {0,1,2}{
\draw (xx) -- (x\j);
\draw (yy) -- (y\j);
\draw (zz) -- (z\j);
\draw (x\j) -- (x0\j);
\draw (x\j) -- (x1\j);
\draw (x1\j) -- (x10\j);
\draw (x1\j) -- (x11\j);
\draw (y\j) -- (y0\j);
\draw (z\j) -- (z0\j);
\draw[ arc,black, ] (x0\j) to[bend right = 0]  (y\j);
\draw[ arc,gray, ] (x0\j) to[bend right = 0]  (z\j);
\draw[arc, , cyan] let \n1 = {int(mod(\j+1,3))} in (x10\j) to[bend right = 0] 
		(y0\n1);
\draw[arc, , red]let \n1 = {int(mod(\j+2,3))} in (x10\j) to[bend left = 0] 
	(z0\n1);
	\draw[arc, , blue] let \n1 = {int(mod(\j+2,3))} in (x11\j) to[bend left = 0] 
	(y0\n1);
		\draw[arc, , orange] let \n1 = {int(mod(\j+1,3))} in (x11\j) to[bend left = 0] 
	(z0\n1);
	}
	
\end{tikzpicture}
}
\caption{The graph $(G_{8}; 3)$ with $\alpha = \delta=1$ and $\beta = \gamma = 2$ is the Tutte 8-cage}
\label{fig:g8Txy}
\end{figure}
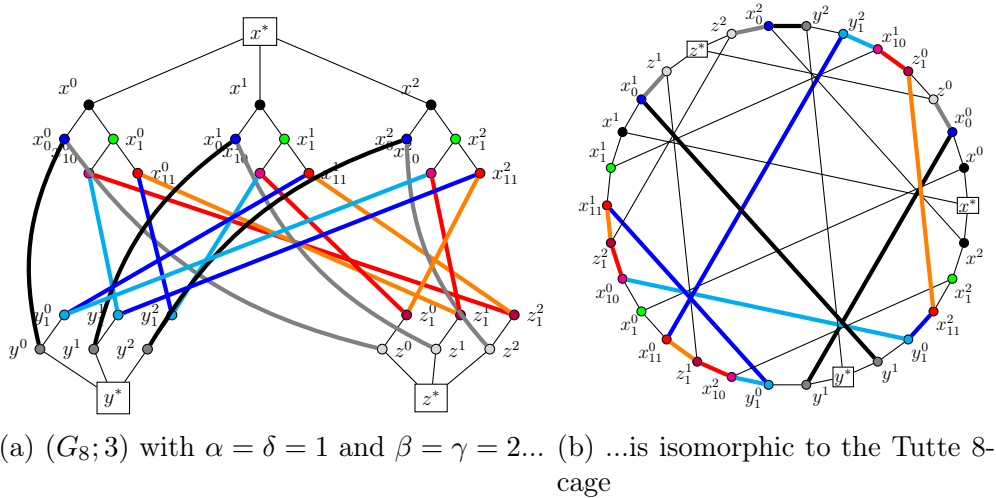

We will give a voltage assignment of $G_{4t}$ for $t=2$ and we will prove that we can obtain a family of semicubic graphs of girth $8$. These graphs are a generalization of the graphs given in \cite{ABLM13}; specifically, they are obtained if, in the following Theorem, we have that $\alpha=3$, $\beta=1$, $\gamma=2 $, and $\delta=3$. 
\begin{theorem}\label{thm:g8}
The  graphs $(G_{8}; m)$ given by the voltage graph $G_8$ with voltage and arc assignment $(x_{10}y_{1})=\alpha$, $(x_{10}z_{1})=\beta$, $(x_{11}y_{1})=\gamma$ and $(x_{11}z_{1})=\delta$ with $\alpha\neq \beta\neq 0$, $\gamma\neq \delta\neq 0$ and $(\alpha-\beta)-(\gamma-\delta)\neq 0 \mod{m}$  has girth $8$ for $m\geq 3$. Moreover, it is a $(\{3,m\};8)$-graph with three vertices of degree $m$ and $9m$ vertices of degree $3$.
\end{theorem}
\begin{proof}First, note that by Observation \ref{lem:nonCycleWalks}, since $G_{8}$ has no 2-cycles, there are no closed non-cycle non-reversing walks that lift to cycles of length 4 or 6. Thus,
since $(G_8; m)$ is bipartite, it is enough to show that every cycle of length $4$ or $6$ in $G_8$ lifts to a cycle of length at least $8$ in $(G_8; m)$. 

We enumerated all (directed) cycles of length 4 and 6 in $G_{8}$ directly, reversing arrows as needed, and summed their voltages: there are 18 such cycles, six 4-cycles and 12 6-cycles. For example, 
 $\overunderbraces{&\br{3}{\beta}& &\br{3}{\gamma}}%
    {(&x_{10}&, &z_{0}&,& x_{11}&, &y_{0}&, & x_{10}&)}%
    {& &&\br{3}{-\delta}& &\br{3}{-\alpha}}$
is a 4-cycle with voltage sum $-\alpha + \beta +\gamma-\delta$.

Summing the voltages in each cycle, we observed that the collection of possible voltage sums is the following set: $\{-\alpha +\beta +\gamma -\delta
   ,\alpha -\beta -\gamma +\delta
   ,\delta -\beta ,\gamma -\alpha
   ,\beta -\delta ,\alpha -\gamma
   ,\gamma -\delta ,\alpha -\beta
   ,\delta -\gamma ,\beta -\alpha
   ,\delta ,\gamma ,\beta ,\alpha
   ,-\delta ,-\beta ,-\gamma
   ,-\alpha \}$. Thus,  if $\alpha\neq \beta\neq 0$, $\gamma\neq \delta\neq 0$ and $(\alpha-\beta)-(\gamma-\delta)\neq 0 \mod{m}$, then these short cycles do not lift to short cycles in $(G_{8}, m)$.  
By Lemma \ref{txygirth8}, we have that the girth of $(G_{8}, m)$ is at most 8. Observe that, for example, $ (x^*, x^0, x_{0}^0, y^0,  y^*, y^1, x_{0}^1, x^1, x^{*})$ is an 8-cycle. Thus, since $(G_{8}, m)$ has no cycles of length less than 8, 
 the girth of $G_{8}$ is exactly 8.
\end{proof}
In particular,  $\alpha=\delta=1$ and $\beta=\gamma=2$ produces a family of graphs of girth $8$. \\

\subsubsection{Two families of semicubic graphs of girth $12$}
Now, we will give a voltage assignment to produce a graph $G_{4t}$ for $t=3$ that lifts to a family of semicubic graphs of girth $12$. The additional arcs and voltages are listed in Table \ref{table:G12table} and the graph itself is shown in Figure \ref{fig:g12}.

These arc assignments and voltages were found by trial-and-error, with the assumption that each leaf-vertex in $X'_{3}$ needed to be connected to one vertex in $Y_{2}$ and one vertex in $Z_{2}$ in such a way as to avoid creating 4-cycles. 
\begin{theorem}\label{girth12voltage}
The  graph $(G_{12}; m)\in \mathcal{G}_{4t}$ for $t=3$ given by the voltage graph $G_{12}$ formed by adding the arcs and voltage assignments to $T_{12}$ given in Table \ref{table:G12table} 
has girth $12$ for $m\geq 9$. 
Moreover, it is a $(\{3,m\};12)$-graph with three vertices of degree $m$ and $49m$ vertices of degree $3$. 
\end{theorem}

\begin{table}[htp]
\caption{Labelled arcs between leaves of $T_{12}$ forming a voltage graph $G_{12}$. The $\mathbb{Z}_{m}$ lifts $(G_{12}, m)$ are girth 12 for $m\geq 9$.} 
\begin{center}


{\small
\begin{tabular}{c|  c  c c c c c c c c c c c}
Starting leaf & $x_{0100}$ & $x_{0101}$ & 
$x_{0110}$ & $x_{0111}$ & $x_{1000}$ & $x_{1001}$ & $x_{1010}$ & $x_{1001}$ 
 & $x_{1100}$ & $x_{1101}$ & $x_{1110}$ & $x_{1111}$ \\
Ending leaf & $y_{010}$ & $y_{011}$ & $y_{100}$ & $ y_{101}$ & $y_{010}$& $y_{011}$ & $y_{110}$& $y_{111}$ & $y_{100}$& $y_{101}$ & $y_{110}$& $y_{111}$\\
Voltage & 1 & 2 & 1 & 2 & 2 & 1 & 1 & 2 & 2 & 1 & 3 & -1 \\ \hline
Starting leaf & $x_{0100}$ & $x_{0101}$ & 
$x_{0110}$ & $x_{0111}$ & $x_{1000}$ & $x_{1001}$ & $x_{1010}$ & $x_{1001}$ 
 & $x_{1100}$ & $x_{1101}$ & $x_{1110}$ & $x_{1111}$ \\
Ending leaf & $z_{110}$ & $z_{111}$ & $z_{100}$ & $ z_{101}$ & $z_{010}$& $z_{011}$ & $z_{100}$& $z_{101}$ & $z_{010}$& $z_{011}$ & $z_{110}$& $z_{111}$\\
Voltage & 2 & 1 & -1 & 3 & -1 & 3 & 1 & -1 & -3 & -2 & -1 & 2 \\ 

\end{tabular}
}
\end{center}
\label{table:G12table}
\end{table}%

\begin{proof}
Since $(G_{12}; m)$ is bipartite, it is enough to show that every even circuit in the voltage graph $G_{12}$ of length smaller than $12$ lifts to a cycle of length at least $12$ in $(G_{12}; m)$. It is easy to find 12-cycles in $(G_{12}; m)$ so the girth is at most 12.

As before, we analyzed all cycles of length at most 10 (``short cycles'' in the voltage graph), using \emph{Mathematica} \cite{Mathematica}, considering the directed voltage graph in which all single-directed arcs are replaced with a pair of alternately oriented arcs (labeling with the negative voltage for the oppositely oriented arc) to allow traveling along the arc in either direction.

There are $252$ short cycles that are not formed by going back and forth along a single edge. We summed the voltages along each (directed) cycle, and observed that the absolute value of the voltage sums along the cycles all lie in the set $\{1, 2, 3, 4, 5, 6, 7, 8\}$. Thus, when $m \leq 8$, there is a short cycle in the voltage graph that lifts to a short cycle in the lift graph, dropping the girth for those values of $m$, but for all other values of $m$, no short cycle in the voltage graph lifts to a short cycle in $(G_{12};m)$.

By construction of $T_{12}$, all lollipop walks lift to long cycles. Next, we need to consider non-reversing non-cycle walks in the voltage graph that may lift to short cycles in the derived graph. From Observation \ref{lem:nonCycleWalks}, any such walk must involve at least one 4-cycle. However, analysis of the cycles in $G_{12}$ shows that, in fact, there are no 4-cycles (by construction), so there are no short walks of this sort that lift to short cycles.

\end{proof}

\begin{figure}[htbp]
\begin{center}
\begin{tikzpicture}[vtx/.style={draw, circle, inner sep = 1 pt, font = \tiny},bigvtx/.style={draw, , inner sep = 1 pt, font = \tiny, minimum height = 10 pt},lbl/.style = {midway, fill = white, inner sep = 1 pt, font=\tiny},]
\pgfmathsetmacro{\scale}{2}
\pgfmathsetmacro{\xx}{40*\scale}
\pgfmathsetmacro{\yy}{15}
\pgfmathsetmacro{\r}{20}
\pgfmathsetmacro{\rr}{8}
\pgfmathsetmacro{\rrr}{2}

\node[bigvtx] (xx) {$x^{*}$};
\node[vtx, below = \yy pt of xx] (x){};
\node[vtx, below left = \yy pt and \xx pt of x] (x0) {};
\node[vtx, below right = \yy pt and \xx pt of x] (x1){}; 
\node[vtx, below left = \yy pt and \xx/\rrr pt  of x0,] (x00) {};
\node[vtx, below right = \yy pt and \xx/\rrr pt  of x0] (x01) {};
\node[vtx, below left = \yy pt and \xx/\rrr pt  of x1] (x10) {};
\node[vtx, below right = \yy pt and \xx/\rrr pt  of x1] (x11) {};

\node[vtx, below left = \yy pt and \xx/\rr pt  of x01] (x010) {};
\node[vtx, below left = \yy pt and \xx/\rr pt  of x10] (x100) {};
\node[vtx, below left = \yy pt and \xx/\rr pt  of x11] (x110) {};

\node[vtx, below right = \yy pt and \xx/\rr pt  of x01] (x011) {};
\node[vtx, below right = \yy pt and \xx/\rr pt  of x10] (x101) {};
\node[vtx, below right = \yy pt and \xx/\rr pt  of x11] (x111) {};

\node[vtx, below left = \yy pt and \xx/\r pt  of x010] (x0100) {};
\node[vtx, below left = \yy pt and \xx/\r pt  of x100] (x1000) {};
\node[vtx, below left = \yy pt and \xx/\r pt  of x110] (x1100) {};
\node[vtx, below left = \yy pt and \xx/\r pt  of x011] (x0110) {};
\node[vtx, below left = \yy pt and \xx/\r pt  of x101] (x1010) {};
\node[vtx, below left = \yy pt and \xx/\r pt  of x111] (x1110) {};
\node[vtx, below right = \yy pt and \xx/\r pt  of x010] (x0101) {};
\node[vtx, below right = \yy pt and \xx/\r pt  of x100] (x1001) {};
\node[vtx, below right = \yy pt and \xx/\r pt  of x110] (x1101) {};
\node[vtx, below right = \yy pt and \xx/\r pt  of x011] (x0111) {};
\node[vtx, below right = \yy pt and \xx/\r pt  of x101] (x1011) {};
\node[vtx, below right = \yy pt and \xx/\r pt  of x111] (x1111) {};

\draw (xx) -- (x);
\foreach \i in {0,1}{
\draw (x) -- (x\i);}
\foreach \j in {0,1}{
\draw (x1) -- (x1\j);
\draw (x10) -- (x10\j);
\draw (x01) -- (x01\j);
\draw (x11) -- (x11\j);

\draw[] (x010) -- (x010\j);
\draw[] (x011) -- (x011\j);
\draw[] (x100) -- (x100\j);
\draw[] (x101) -- (x101\j);
\draw[] (x110) -- (x110\j);
\draw[] (x111) -- (x111\j);
}
\draw (x0) -- (x01);
\draw[] (x0)-- (x00);

\pgfmathsetmacro{\sh}{100}
\pgfmathsetmacro{\ysh}{-200}
\pgfmathsetmacro{\xx}{25*\scale}
\pgfmathsetmacro{\yy}{-13}
\pgfmathsetmacro{\rr}{4}
\pgfmathsetmacro{\rrr}{10}

\begin{scope}[yshift = \ysh, xshift = -\sh]

\node[bigvtx, ] (yy) {$y^{*}$};
\node[vtx, below = 1.5*\yy pt of yy] (y){};
\node[vtx, below left = \yy pt and \xx pt of y] (y0) {};
\node[vtx, below right = \yy pt and \xx pt of y] (y1){}; 
\node[vtx, below right = \yy pt and \xx/\rr pt  of y0] (y01) {};
\node[vtx, below left = \yy pt and \xx/\rr pt  of y1] (y10) {};
\node[vtx, below right = \yy pt and \xx/\rr pt  of y1] (y11) {};

\node[vtx, below left = \yy pt and \xx/\rrr pt  of y01] (y010) {};
\node[vtx, below left = \yy pt and \xx/\rrr pt  of y10] (y100) {};
\node[vtx, below left = \yy pt and \xx/\rrr pt  of y11] (y110) {};

\node[vtx, below right = \yy pt and \xx/\rrr pt  of y01] (y011) {};
\node[vtx, below right = \yy pt and \xx/\rrr pt  of y10] (y101) {};
\node[vtx, below right = \yy pt and \xx/\rrr pt  of y11] (y111) {};

\draw (yy) -- (y);
\foreach \i in {0,1}{
\draw (y) -- (y\i);}
\foreach \j in {0,1}{
\draw (y1) -- (y1\j);
\draw (y10) -- (y10\j);
\draw (y01) -- (y01\j);
\draw (y11) -- (y11\j);
}
\draw (y0) -- (y01);

\end{scope}

\begin{scope}[yshift = \ysh, xshift = \sh]

\node[bigvtx, ] (zz) {$z^{*}$};
\node[vtx, below = 1.5*\yy pt of zz] (z){};
\node[vtx, below left = \yy pt and \xx pt of z] (z0) {};
\node[vtx, below right = \yy pt and \xx pt of z] (z1){}; 
\node[vtx, below right = \yy pt and \xx/\rr pt  of z0] (z01) {};
\node[vtx, below left = \yy pt and \xx/\rr pt  of z1] (z10) {};
\node[vtx, below right = \yy pt and \xx/\rr pt  of z1] (z11) {};

\node[vtx, below left = \yy pt and \xx/\rrr pt  of z01] (z010) {};
\node[vtx, below left = \yy pt and \xx/\rrr pt  of z10] (z100) {};
\node[vtx, below left = \yy pt and \xx/\rrr pt  of z11] (z110) {};

\node[vtx, below right = \yy pt and \xx/\rrr pt  of z01] (z011) {};
\node[vtx, below right = \yy pt and \xx/\rrr pt  of z10] (z101) {};
\node[vtx, below right = \yy pt and \xx/\rrr pt  of z11] (z111) {};

\draw (zz) -- (z);
\foreach \i in {0,1}{
\draw (z) -- (z\i);}
\foreach \j in {0,1}{
\draw (z1) -- (z1\j);
\draw (z10) -- (z10\j);
\draw (z01) -- (z01\j);
\draw (z11) -- (z11\j);
}
\draw (z0) -- (z01);

\end{scope}

\draw[red, ->] (x0100) --node[lbl, pos = .2]{1} (y010);
\draw[red, ->] (x0100) to[out = -35, in = 180-20]node[lbl, pos = .23]{2} (z110);

\draw[cyan,densely dotted,thick, ->] (x0101) --node[lbl, pos =.2]{2} (y011);
\draw[cyan, densely dotted,thick,->] (x0101) to[out = -35, in = 180-20]node[lbl, pos = .21]{1} (z111);
\draw[red, ->] (x0110) --node[lbl, near end]{1} (y100);
\draw[red, ->] (x0110) --node[lbl, pos = .25]{-1} (z100);

\draw[cyan,densely dotted, thick, ->] (x0111) --node[lbl,pos = .6]{2} (y101);
\draw[cyan,densely dotted,thick, ->] (x0111) --node[lbl, pos = .22]{3} (z101);
\draw[red, ->] (x1000) --node[lbl, pos = .7]{2} (y010);
\draw[red, ->] (x1000) --node[lbl, pos = .8]{-1} (z010);

\draw[cyan,densely dotted,thick, ->] (x1001) --node[lbl, pos = .7]{1} (y011);
\draw[cyan,densely dotted,thick, ->] (x1001) --node[lbl, pos = .26]{3} (z011);

\draw[red, ->] (x1010) --node[lbl, pos = .55]{1} (y110);
\draw[red, ->] (x1010) --node[lbl, pos = .61]{1} (z100);

\draw[cyan,densely dotted,thick, ->] (x1011) --node[lbl, pos = .1]{2} (y111);
\draw[cyan,densely dotted,thick, ->] (x1011) --node[lbl, pos = .6]{-1} (z101);

\draw[red, ->] (x1100) --node[lbl, pos =.3]{2} (y100);
\draw[red, ->] (x1100) --node[lbl, pos =.5]{-3} (z010);

\draw[cyan, densely dotted,thick,->] (x1101) --node[lbl,pos =.31]{1} (y101);
\draw[cyan,densely dotted,thick, ->] (x1101) --node[lbl, pos =.31]{-2} (z011);

\draw[red, ->] (x1110) --node[lbl, pos =.4]{3} (y110);
\draw[red, ->] (x1110) --node[lbl, pos =.21]{-1} (z110);

\draw[cyan,densely dotted, thick,->] (x1111) --node[lbl, pos =.41]{-1} (y111);
\draw[cyan, densely dotted,thick,->] (x1111) --node[lbl, pos =.2]{2} (z111);

\draw[ thick] (x00) to[bend right = 20] (y0);
\draw[ thick] (x00) to[bend right = 17] (z0);
\end{tikzpicture}
\caption{The voltage graph $G_{12}$. The colors/styles on the labeled arcs show the two 12-cycles (appropriately directing the arrows) formed by the labeled arcs between the leaves of $T_{12}$.}\label{fig:g12}
\end{center}
\end{figure}

The smallest element of this family is a semiregular graph with 3 vertices of degree 9 and $441 = 49\cdot 9$ of vertices of degree 3.

As in Section \ref{case4t+2}, we modified the graph $T_{12}$ and added additional labeled arcs to produce a new voltage graph $H_{12}$ with fewer vertices than $G_{12}$. The $\mathbb{Z}_{m}$ lifts of $H_{12}$ produce graphs of girth 12 with $41m + 3$ vertices, which is a significant improvement. Specifically, we modified $T_{12}$ by deleting the eight leaf vertices $x_{1000}$, $x_{1001}$, $x_{1100}$ $x_{1101}$, $y_{100}$, $y_{111}$, $z_{100}$, $z_{111}$ and incident edges. These deleted vertices and edges are shown in light gray in Figure \ref{fig:H12}. We then added new labeled directed arcs from $x_{100}$ to $y_{10}, z_{10}$ and from $x_{110}$ to $y_{111}$ and $z_{111}$, and then searched for voltage assignments on those arcs and arcs added between the remaining leaves of $T_{12}$ that would lift to graphs of girth 12. 

\begin{figure}
\begin{center}
\begin{tikzpicture}[vtx/.style={draw, circle, inner sep = 1 pt, font = \tiny},bigvtx/.style={draw, , inner sep = 1 pt, font = \tiny, minimum height = 10 pt},lbl/.style = {midway,  fill = white, inner sep = 1 pt, font=\tiny},]
\pgfmathsetmacro{\scale}{2}
\pgfmathsetmacro{\xx}{40*\scale}
\pgfmathsetmacro{\yy}{15}
\pgfmathsetmacro{\r}{20}
\pgfmathsetmacro{\rr}{8}
\pgfmathsetmacro{\rrr}{2}

\node[bigvtx] (xx) {$x^{*}$};
\node[vtx, fill=black, below = \yy pt of xx] (x){};
\node[vtx, below left = \yy pt and \xx pt of x] (x0) {};
\node[vtx, below right = \yy pt and \xx pt of x] (x1){}; 
\node[vtx, fill=black,below left = \yy pt and \xx/\rrr pt  of x0,] (x00) {};
\node[vtx,fill=black, below right = \yy pt and \xx/\rrr pt  of x0] (x01) {};
\node[vtx,fill=black, below left = \yy pt and \xx/\rrr pt  of x1] (x10) {};
\node[vtx,fill=black, below right = \yy pt and \xx/\rrr pt  of x1] (x11) {};

\node[vtx, below left = \yy pt and \xx/\rr pt  of x01] (x010) {};
\node[vtx, below left = \yy pt and \xx/\rr pt  of x10] (x100) {};
\node[vtx, below left = \yy pt and \xx/\rr pt  of x11] (x110) {};

\node[vtx, below right = \yy pt and \xx/\rr pt  of x01] (x011) {};
\node[vtx, below right = \yy pt and \xx/\rr pt  of x10] (x101) {};
\node[vtx, below right = \yy pt and \xx/\rr pt  of x11] (x111) {};

\node[vtx, fill=black,below left = \yy pt and \xx/\r pt  of x010] (x0100) {};
\node[vtx,fill=black, below left = \yy pt and \xx/\r pt  of x100, opacity = .15] (x1000) {};
\node[vtx, fill=black,below left = \yy pt and \xx/\r pt  of x110, opacity = .15] (x1100) {};
\node[vtx,fill=black, below left = \yy pt and \xx/\r pt  of x011] (x0110) {};
\node[vtx,fill=black, below left = \yy pt and \xx/\r pt  of x101] (x1010) {};
\node[vtx,fill=black, below left = \yy pt and \xx/\r pt  of x111] (x1110) {};
\node[vtx,fill=black, below right = \yy pt and \xx/\r pt  of x010] (x0101) {};
\node[vtx, fill=black,below right = \yy pt and \xx/\r pt  of x100, opacity = .15] (x1001) {};
\node[vtx,fill=black, below right = \yy pt and \xx/\r pt  of x110, opacity = .15] (x1101) {};
\node[vtx,fill=black, below right = \yy pt and \xx/\r pt  of x011] (x0111) {};
\node[vtx,fill=black, below right = \yy pt and \xx/\r pt  of x101] (x1011) {};
\node[vtx,fill=black, below right = \yy pt and \xx/\r pt  of x111] (x1111) {};

\draw (xx) -- (x);
\foreach \i in {0,1}{
\draw (x) -- (x\i);}
\foreach \j in {0,1}{
\draw (x1) -- (x1\j);
\draw (x10) -- (x10\j);
\draw (x01) -- (x01\j);
\draw (x11) -- (x11\j);

\draw[] (x010) -- (x010\j);
\draw[] (x011) -- (x011\j);
\draw[opacity = .15] (x100) -- (x100\j);
\draw[] (x101) -- (x101\j);
\draw[opacity = .15] (x110) -- (x110\j);
\draw[] (x111) -- (x111\j);
}
\draw (x0) -- (x01);
\draw[] (x0)-- (x00);

\pgfmathsetmacro{\sh}{100}
\pgfmathsetmacro{\ysh}{-200}
\pgfmathsetmacro{\xx}{25*\scale}
\pgfmathsetmacro{\yy}{-13}
\pgfmathsetmacro{\rr}{4}
\pgfmathsetmacro{\rrr}{10}

\begin{scope}[yshift = \ysh, xshift = -\sh]

\node[bigvtx, ] (yy) {$y^{*}$};
\node[vtx,fill=black, below = 1.5*\yy pt of yy] (y){};
\node[vtx, below left = \yy pt and \xx pt of y] (y0) {};
\node[vtx, below right = \yy pt and \xx pt of y] (y1){}; 
\node[vtx,fill=black, below right = \yy pt and \xx/\rr pt  of y0] (y01) {};
\node[vtx,fill=black, below left = \yy pt and \xx/\rr pt  of y1] (y10) {};
\node[vtx,fill=black, below right = \yy pt and \xx/\rr pt  of y1] (y11) {};

\node[vtx, below left = \yy pt and \xx/\rrr pt  of y01] (y010) {};
\node[vtx, below left = \yy pt and \xx/\rrr pt  of y10, opacity = .15] (y100) {};
\node[vtx, below left = \yy pt and \xx/\rrr pt  of y11] (y110) {};

\node[vtx, below right = \yy pt and \xx/\rrr pt  of y01] (y011) {};
\node[vtx, below right = \yy pt and \xx/\rrr pt  of y10] (y101) {};
\node[vtx, below right = \yy pt and \xx/\rrr pt  of y11, opacity = .15] (y111) {};

\draw (yy) -- (y);
\foreach \i in {0,1}{
\draw (y) -- (y\i);}
\draw[opacity = .15] (y10) -- (y100);
\draw (y10) -- (y101);
\foreach \j in {0,1}{
\draw (y1) -- (y1\j);
\draw (y01) -- (y01\j);
}
\draw (y0) -- (y01);
\draw[opacity=.15] (y11) -- (y111);
\draw[] (y11) -- (y110);

\end{scope}

\begin{scope}[yshift = \ysh, xshift = \sh]

\node[bigvtx, ] (zz) {$z^{*}$};
\node[vtx,fill=black, below = 1.5*\yy pt of zz] (z){};
\node[vtx, below left = \yy pt and \xx pt of z] (z0) {};
\node[vtx, below right = \yy pt and \xx pt of z] (z1){}; 
\node[vtx,fill=black, below right = \yy pt and \xx/\rr pt  of z0] (z01) {};
\node[vtx,fill=black, below left = \yy pt and \xx/\rr pt  of z1] (z10) {};
\node[vtx,fill=black, below right = \yy pt and \xx/\rr pt  of z1] (z11) {};

\node[vtx, below left = \yy pt and \xx/\rrr pt  of z01] (z010) {};
\node[vtx, below left = \yy pt and \xx/\rrr pt  of z10, opacity = .15] (z100) {};
\node[vtx, below left = \yy pt and \xx/\rrr pt  of z11] (z110) {};

\node[vtx, below right = \yy pt and \xx/\rrr pt  of z01] (z011) {};
\node[vtx, below right = \yy pt and \xx/\rrr pt  of z10] (z101) {};
\node[vtx, below right = \yy pt and \xx/\rrr pt  of z11, opacity = .15] (z111) {};

\draw (zz) -- (z);
\foreach \i in {0,1}{
\draw (z) -- (z\i);}
\foreach \j in {0,1}{
\draw (z1) -- (z1\j);
\draw (z01) -- (z01\j);
}
\draw (z10) -- (z101);
\draw[opacity = .15] (z10) -- (z100);
\draw (z0) -- (z01);
\draw (z11) -- (z110);
\draw[opacity=.15] (z11) -- (z111);

%

\end{scope}

\draw[orange, thick, ->] (x100) to[bend left = 10]node[lbl, pos = .2,rotate = 53]{$a=1$} (y10);
\draw[orange, thick, ->] (x100) to[out = -50, in = 180]node[lbl, pos = .68,rotate = -33]{$b=-1$} (z10);

\draw[green!60!black, thick, ->] (x110) to[out = 180+50, in = 25]node[lbl, pos = .2, rotate = 43]{$c=-1$} (y11);
\draw[green!60!black, thick, ->] (x110) to[out = -90+20, in = 30]node[lbl, pos = .65, rotate = -50]{$d=1$} (z11);


\draw[magenta, ->] (x0100) to[bend right = 10]node[lbl, pos = .8,rotate = 40]{$e=2$} (y010);
\draw[magenta, ->] (x0100) to[in = 180-20, out = -30]node[lbl, pos = .8, rotate = -15]{$f=1$} (z101);
\draw[magenta, ->] (x1111) to[in = 30, out = 180+45]node[lbl, pos = .1, rotate = 25]{$g=1$} (y010);
\draw[magenta, ->] (x1111) to[bend left = 10]node[lbl, pos = .75,rotate = -85]{$h=2$} (z110);


\draw[magenta, ->] (x0101) to[bend right = 10]node[lbl, pos = .6,rotate = 40]{$j=1$} (y011);
\draw[magenta, ->,] (x0101) to[in = 180-25, out = -30]node[lbl, pos = .75, rotate = -15]{$k=-2$} (z110);
\draw[magenta, ->, ] (x1110) to[in = 30,out = 180+45]node[lbl, pos = .1,rotate = 30]{$l=2$} (y011);
\draw[magenta, ->, ] (x1110) to[bend left = 10]node[lbl, pos = .7, rotate = -85]{$p=-2$} (z101);

\draw[blue, ->, , densely dotted, thick] (x0110) to[bend right = 10]node[lbl, pos = .3,rotate = 75]{$q=-1$} (y101);
\draw[blue, ->, , densely dotted, thick] (x0110) to[in = 180-30, out = -30]node[lbl, pos = .85, rotate = -30]{$r=1$} (z010);
\draw[blue, ->, , densely dotted, thick] (x1011) to[in = 30,out = 180+45]node[lbl, pos = .18, rotate = 35]{$s=-3$} (y101);
\draw[blue, ->, , densely dotted, thick] (x1011) to[bend left = 10]node[lbl, pos = .2, rotate = -70]{$t=-6$} (z011);


\draw[blue, ->, densely dotted, thick] (x0111) to[bend right = 10]node[lbl, pos = .4, rotate = 75]{$u=1$} (y110);
\draw[blue, ->, densely dotted, thick] (x0111) to[in = 180-30, out = -30]node[lbl, pos = .77,rotate = -30]{$v=-1$} (z011);
\draw[blue, ->, densely dotted, thick] (x1010) to[in = 30,out = 180+45]node[lbl, pos = .18, rotate = 35]{$w=-2$} (y110);
\draw[blue, ->, densely dotted, thick] (x1010) to[bend left = 20]node[lbl, pos = .6, rotate = -90]{$\alpha=2$} (z010);


\draw[ thick] (x00) to[bend right = 20] (y0);
\draw[ thick] (x00) to[bend right = 17] (z0);
\end{tikzpicture}
\caption{The voltage graph $H_{12}$. The colors/styles of the arcs between the leaves of $T_{12}$ identify disjoint 8-cycles (if arrows are directed appropriately). The leaves deleted from $T_{12}$ to allow the formation of $H_{12}$ are shown in light gray.}
\label{fig:H12}
\end{center}
\end{figure}

\begin{theorem}\label{girth12voltage-H}
The  graph $(H_{12}; m)$ given by the voltage graph $H_{12}$ shown 
in Figure \ref{fig:H12} with voltages
given in Table \ref{table:H12table} 
has girth $12$ for $m \geq 10$. Moreover, it is a $(\{3,m\};12)$-graph with three vertices of degree $m$ and $41m$ vertices of degree $3$. 
\end{theorem}

\begin{table}[htp]
\caption{Labelled arcs between leaves of a modified $T_{12}$ forming a voltage graph $H_{12}$. The $\mathbb{Z}_{m}$ lifts $(H_{12}, m)$ are girth 12 for $m \geq 10$.}
\begin{center}

\bigskip

%

{\small
\begin{tabular}{c | c c c c  }
Starting leaf & $x_{100}$ & $x_{100}$ & $x_{110}$ & $x_{110}$\\
Ending leaf & $y_{10}$ & $z_{10}$ & $y_{11}$ & $z_{11}$\\
Voltage & $a=1$ & $b=-1$ & $c=-1$ & $d=1$
\end{tabular}

\bigskip

\begin{tabular}{c|  c  c c c c c c c}
Starting leaf & $x_{0100}$ & $x_{0101}$ & $x_{0110}$ & $x_{0111}$  & $x_{1010}$ & $x_{1001}$ & $x_{1110}$ & $x_{1111}$ \\
Ending leaf & $y_{010}$ & $y_{011}$ & $y_{101}$ & $y_{110}$ & $y_{110}$ & $y_{101}$ & $y_{011}$ & $y_{010}$\\
Voltage & $e=2$ & $j=1$ & $q = -1$ & $u = 1$ & $w=-2$ & $s=-3$ & $l=2$ & $g=1$  \\ \hline
Starting leaf & $x_{0100}$ & $x_{0101}$ & $x_{0110}$ & $x_{0111}$ & $x_{1010}$ & $x_{1001}$  & $x_{1100}$ & $x_{1101}$  \\
Ending leaf & $z_{101}$ & $z_{110}$ & $z_{010}$ & $z_{011}$ & $z_{010}$ & $z_{011}$ & $z_{101}$ & $z_{110}$ \\
Voltage &$f=1$ & $k=-2$ & $r = 1$ & $v = 5$ & $\alpha=2$ & $t=-6$ & $p=-2$ & $h=2$  \\ 
\end{tabular}
}
\end{center}
\label{table:H12table}
\end{table}%

\begin{proof} It is easy to find 12-cycles in $(H_{12}; m)$ so the girth is at most 12. As previously, we use Mathematica to analyze the voltage sums along short cycles (length at most 10) in the graph. By construction, the graph has no 4-cycles, so all non-cycle closed walks lift to cycles of length at least 12, by Observation \ref{lem:nonCycleWalks}. Moreover, all lollipop walks lift to cycles of at least length 12, by construction. (For example, there is a new lollipop walk formed by connecting the path $(x^{*}, x_{1})$ to the cycle $(x_{1}, x_{10}, x_{100}, y_{10}, y_{1}, y_{11}, x_{110}, x_{11}, x_{1})$ that uses two of the newly-introduced arcs that connect vertices that were not leaves of $T_{12}$, but even this lollipop lifts to a cycle of length 12, by Lemma \ref{lollipop}.)

Thus, to determine the girth, it suffices to analyze the voltage sums of all cycles of length 6, 8, 10. There are 254 such cycles. Analyzing the voltage sums along all these cycles as before shows that for each cycle, the absolute value of the voltage sum lies in the set $\{ 1,2,3,4,5,6,7,8,9\}$. Thus, for $1 \leq m \leq 9$, there is a short cycle in the voltage graph that lifts to a short cycle in the derived graph, but for $m > 9$ no short cycle lifts to a short cycle in $(H_{12},m)$.
\end{proof}

\section{Open questions}

Here we give a short, non-comprehensive list of natural open questions related to this work. 
\begin{itemize}

\item Find voltage assignments to construct graphs of girth $10$ and $12$ for missing values of $m$. 
More specifically: 
\begin{itemize}

\item As we said at the end of Section \ref{girth10}, $(H_{10};m)$ is a graph of girth 10 with $20m$ vertices of degree $3$ and only two vertices of degree $m$, for $m\geq 7$. In \cite{AEJ16} the authors give a graph with the same parameters for $m=4$, but this graph has $4$ vertices of degree $m$ and only $78$ vertices of degree $3$. Can we find a similar voltage graph to $H_{10}$ (either a different voltage assignment for the same arcs as in $H_{10}$, or a different assignment of arcs and voltages using the same ``underlying'' tree as in $H_{10}$) that gives us a derived graph $(H'_{10};m)$ of girth 10 for all $m\geq 4$? Preliminary investigations suggest that new arc assignments are required to find smaller examples.

\item Find voltage assignments to produce $(G_{12};m)$ graphs with the construction given in Theorem \ref{girth12voltage} for any $4\leq m \leq 8$, or show no such graphs exist.  
\item Find different voltage assignments for the graph $H_{12}$ that produce graphs of girth 12 for $m \in \{3,4,\ldots, 9\}$,  or show no assignments exist. Preliminary investigations suggest that new arc assignments are required to find smaller examples.
Does there exist a different graph $H'_{12}$ (for example, using the same modified tree structure and arcs $a,b,c,d$ but different arcs among the remaining leaves of $T_{12}$) that produces graphs of girth 12 for these smaller values of $m$?  
 For example, can we find graphs of small girth if we construct a 16-cycle among the remaining leaves of $T_{12}$, or if we interlace two 8-cycles differently?
 \end{itemize}


 \item In Corollary \ref{girth14}, we obtain $(3,m;14)$-graphs of order greater than $115m$ for $m\geq 3$.  Can this voltage method be used to produce graphs of even girth $14$ or larger in a tractable way? The voltage assignments used to construct the graph families exhibited in this paper were found by educated trial-and-error. This approach becomes less tractable as the number of parameters increases.

\item Study the same problem with an odd degree. The last two authors are making progress on this project \cite{ABF23}.

\end{itemize}

\end{document}